\documentclass[11pt,reqno]{amsart}

\usepackage{amsmath}
\usepackage{amsthm}
\usepackage[T1]{fontenc}
\usepackage{mathrsfs}
\usepackage{latexsym}
\usepackage{amssymb}
\usepackage{bm} 
\usepackage{mathabx}
\usepackage{graphicx}
\usepackage{tikz}
\usepackage{float}
\usepackage{extarrows}
\usepackage{epstopdf}
\usepackage{caption}
\usepackage{subcaption}
\usepackage{enumerate}   
\usepackage{hyperref}
\hypersetup{
    colorlinks,
    linkcolor={luh-dark-blue},
    citecolor={luh-dark-blue},
    urlcolor={luh-dark-blue}
}
\usepackage{xcolor}
\definecolor{luh-dark-blue}{rgb}{0.0, 0.313, 0.608}
\definecolor{lred}{rgb}{1.0,0.5,0.5}

\usepackage[a4paper,left=3.5cm, right=3cm,top=3cm, bottom=3cm]{geometry}

\usepackage{color}

\newcommand{\N}{\mathbb{N}}		   
\newcommand{\Z}{\mathbb{Z}}				
\newcommand{\C}{\mathbb{C}}	
\newcommand{\T}{\mathbb{T}}	
\newcommand{\R}{\mathbb{R}}	

\newcommand{\e}{\varepsilon}

\newtheorem{theo}{Theorem}[section]
\newtheorem{cor}[theo]{Corollary}
\newtheorem{lem}[theo]{Lemma}
\newtheorem{prop}[theo]{Proposition}
\newtheorem{rem}[theo]{Remark}
\newtheorem{defi}[theo]{Definition}
\numberwithin{equation}{section}

\DeclareMathOperator{\supp}{supp}

\DeclareMathOperator{\re}{Re}

\usepackage{enumerate}
\usepackage{enumitem}

\title[Waves of maximal height for a class of nonlocal equations]{Waves of maximal height for a class of nonlocal equations with homogeneous symbols}
\author{Gabriele Bruell}
\address{Institute for Analysis, Karlsruher Institute of Technology (KIT), D-76128 Karlsruhe, Germany}
\email{gabriele.bruell@kit.edu}
\author{Raj Narayan Dhara}
\address{Department for Mathematical Sciences, Norwegian University of Science and Technology (NTNU), NO-7491 Trondheim, Norway}
\email{raj.dhara@ntnu.no, r.dhara@mimuw.edu.pl}
\thanks{Date: \today }

\begin{document}
\subjclass[2010]{35B10, 35B32, 35B65, 35S30, 45M15}
\keywords{Highest wave; singular solution; nonlocal equation with homogeneous symbol; fractional KdV equation}
\maketitle 

\begin{abstract}
\noindent
We discuss the existence and regularity of periodic traveling-wave solutions of a class of nonlocal equations with homogeneous symbol of order $-r$, where $r>1$. Based on the properties of the nonlocal convolution operator, we apply analytic bifurcation theory and show that  a  highest, peaked, periodic traveling-wave solution is reached as the limiting case at the end of the main bifurcation curve. The regularity of the highest wave is proved to be exactly Lipschitz. As an application of our analysis, we reformulate the steady reduced Ostrovsky equation in a nonlocal form in terms of a Fourier multiplier operator with symbol $m(k)=k^{-2}$. Thereby we recover its unique highest $2\pi$-periodic, peaked traveling-wave solution, having the property of being exactly Lipschitz at the crest.
\end{abstract}

\noindent

\section{Introduction }

The present study is concerned with the existence and regularity of a highest, periodic traveling-wave solution of the nonlocal equation
\begin{equation}\label{eq:nonlocal}
	u_t + L_ru_x + uu_x=0,
\end{equation}
where $L_r$ denotes the Fourier multiplier operator with symbol $m(k)=|k|^{-r}$, $r>1$.
Equation \eqref{eq:nonlocal} is also known as the \emph{fractional Korteweg--de Vries equation}.
We are looking for $2\pi$-periodic traveling-wave solutions $u(t,x)=\phi(x-\mu t)$, where $\mu>0$ denotes the speed of the right-propagating wave. In this context equation \eqref{eq:nonlocal} reduces after integration to
\begin{equation}\label{eq:steady}
	-\mu \phi + L_r\phi + \frac{1}{2}\phi^2=B,
\end{equation}
where $B\in \R$ is an integration constant.
Since the symbol of $L_r$ is homogeneous, any bounded solution of the above equation has necessarily zero mean; in turn this implies that the integration constant $B$ is uniquely determined to be
\[
	B=\frac{1}{4\pi}\int_{-\pi}^\pi\phi^2(x)\,dx.
\]








\medskip

The question about singular, highest waves was already raised by Stokes. In 1880 Stokes conjectured that the Euler equations admit a highest, periodic traveling-wave having a corner singularity at each crest with an interior angle of exactly $120^\circ$. About 100 years  later (in 1982) Stokes' conjecture was answered in the affirmative by Amick, Fraenkel, and Toland \cite{AFT}.
Subject of a recent investigation by Ehrnstr\"om and Wahl\'en \cite{EW} is the existence and precise regularity of a highest, periodic traveling-wave solution for the Whitham equation; thereby proving  Whitham's conjecture on the existence of such a singular solution. The (unidirectional) Whitham equation is a genuinely nonlocal equation, which can be recovered from the well known Korteweg--de Vries equation by replacing its dispersion relation by one branch of the full Euler dispersion relation. The resulting equation takes (up to a scaling factor) the form of \eqref{eq:nonlocal}, where the symbol of the Fourier multiplier is given by $m(k)=\sqrt{\frac{\tanh(k)}{k}}$. In order to prove their result, Ehrnstr\"om and Wahl\'en developed a general approach based on the regularity and monotonicity properties of the convolution kernel induced by the Fourier multiplier. The highest, periodic traveling-wave solution for the Whitham equation is exactly $C^\frac{1}{2}$-H\"older continuous at its crests; thus exhibiting exactly half the regularity of the highest wave for the Euler equations. In a subsequent paper, Ehrnstr\"om, Johnson, and Claasen \cite{EJC} studied the existence and regularity of a highest wave for the bidirectional Whitham equation incorporating the full Euler dispersion relation leading to a nonlocal equation with cubic nonlinearity and a  Fourier multiplier with symbol $m(k)=\frac{\tanh(k)}{k}$. 
The question addressed in \cite{EJC} is whether this equation gives rise to a highest, periodic, traveling wave, which is peaked (that is, whether it has a corner at each crest), such as the corresponding solution to the Euler equations? Overcoming the additional challenge of the cubic nonlinearity, the authors in \cite{EJC} follow a similar approach as implemented for the Whitham equation in \cite{EW} and prove that the highest wave has a singularity at its crest of the form $|x\log(|x|)|$; thereby still being a cusped wave. Concerning a different model equation arising in the context of shallow-water equations, Arnesen \cite{A} investigated the existence and regularity of a highest, periodic, traveling-wave solution for the Degasperis--Procesi equation. The Degasperis--Procesi equation is a local equation, but it can also be written in a nonlocal form with quadratic nonlinearity and a Fourier multiplier with symbol $m(k)=(1+k^2)^{-1}$, which is acting itself --in contrast to the previously mentioned equations-- on a quadratic nonlinearity. For the Degasperis--Procesi and indeed for all equations in the so-called \emph{b-family} (the famous Camassa--Holm equation being also such a member), explicit peaked, periodic, traveling-wave solutions are known \cite{CH, DHK}.  Using the nonlocal approach introduced originally for the Whitham equation in \cite{EW}, the author of \cite{A} adapts the method to the nonlocal form of the Degasperis--Procesi equation and recovers not only the existence of a highest, peaked, periodic traveling wave, but also proves that any even, periodic, highest wave of the Degasperis--Procesi equation is exactly Lipschitz continuous at each crest; thereby excluding the existence of even, periodic, \emph{cusped} traveling-wave solutions.

\medskip

Of our concern is the existence and regularity of highest, traveling waves for the fractional Korteweg--de Vries equation \eqref{eq:nonlocal}, where $r>1$. In the case when $r=2$, \eqref{eq:nonlocal} can be viewed as the nonlocal form of the  \emph{reduced Ostrovsky equation}
\[
	(u_t+uu_x)_x=u.
\]
For the reduced Ostrovsky equation, a highest, periodic, peaked traveling-wave solution is known explicitly \cite{Ostrovsky1978} and its regularity at each crest is exactly Lipschitz continuous. Recently, the existence and stability of smooth, periodic traveling-wave solutions for the reduced Ostrovsky equation, was investigated in \cite{GP, HSS}. In \cite{GP2}, the authors prove that the (unique) highest, $2\pi$-periodic traveling-wave solutions of the reduced Ostrovsky equation is linearly and nonlinearly unstable. We are going to investigate the existence and precise regularity of highest, periodic traveling-wave solutions of the entire family of equations $\eqref{eq:nonlocal}$ for Fourier multipliers $L_r$, where $r>1$. Based on the nonlocal approach introduced for the Whitham equation \cite{EW}, we adapt the method in a way which is convenient to treat homogeneous symbols, and prove the existence and precise Lipschitz regularity of highest, periodic, traveling-wave solutions of \eqref{eq:nonlocal} corresponding to the symbol $m(k)=|k|^{-r}$, where $r>1$. The advantage of this nonlocal approach relies not only in the fact that it can be applied to various equations of local and nonlocal type, but in particular, that it is suitable to study entire families of equations simultaneously; thereby providing an insight into the interplay between a certain nonlinearity and varying order of linearity.
The main novelty in our work relies upon implementing the approach used in \cite{EW, EJC, A} for equations exhibiting \emph{homogeneous} symbols. For a homogeneous symbol, the associated convolution kernel can not be identified with a positive, decaying function on the real line. Instead we have to work with a periodic convolution kernel. The lack of positivity of the kernel can be compensated by working within the class of zero mean function, though. Moreover, we affirm that starting with a linear operator of order strictly smaller than $-1$ in equation \eqref{eq:nonlocal} a further decrease of order does not affect the regularity of the corresponding  highest, periodic  traveling-wave.

\subsection{Main result and outline of the paper}

Let us formulate our main theorem, which provides the existence of a global bifurcation branch of nontrivial, smooth, periodic and even traveling-wave solutions of equation \eqref{eq:nonlocal}, which reaches a limiting peaked, precisely Lipschitz continuous, solution at the end of the bifurcation curve.

\begin{theo}[Main theorem]\label{thm:main}
	For each integer $k\geq 1$ there exists a wave speed $\mu^*_{k}>0$ and a global bifurcation branch
	\[
	s\mapsto (\phi_{k}(s),\mu_{k}(s)),\qquad s>0,
	\]
	of nontrivial, $\frac{2\pi}{k}$-periodic, smooth, even solutions to the steady equation \eqref{eq:steady} for $r>1$, emerging from the bifurcation point $(0,\mu^*_{k})$. Moreover, given any unbounded sequence $(s_n)_{n\in\N}$ of positive numbers $s_n$, there exists a subsequence of $(\phi_{k}(s_n))_{n\in \N}$, which converges uniformly to a limiting traveling-wave solution $(\bar \phi_{k},\bar\mu_{k})$ that solves \eqref{eq:steady} and satisfies
	\[
	\bar \phi_{k}(0)=\bar \mu_{k}.
	\]
	The limiting wave is strictly increasing on $(-\frac{\pi}{k},0)$ and exactly Lipschitz at $x\in \frac{2\pi}{k}\Z$.
\end{theo}
 It is worth to notify that the regularity of peaked traveling-wave solutions is Lipschitz for \emph{all} $r>1$. The reason mainly relies in the smoothing properties of the Fourier multiplier, which is of order strictly bigger than $1$, see Theorem~\ref{thm:regularity}. 
 
 \medskip

The outline of the paper is as follows:
In Section \ref{S:Setting} we introduce the functional-analytic setting, notations, and some general conventions. Properties of general Fourier multipliers with homogeneous symbol and a representation formula for the corresponding convolution kernel are discussed in Section \ref{S:Fourier}.
Section \ref{S:Properties} is the heart of the present work, where we use the regularity and monotonicity properties of the convolution kernel to study a priori properties of bounded, traveling wave solutions of \eqref{eq:nonlocal}. In particular, we prove that an even, periodic traveling-wave solution $\phi$, which is monotone on a half period and whose maximum equals the wave speed, is precisely Lipschitz continuous. Eventually, in Section \ref{S:Global} we investigate the global bifurcation result. By excluding certain alternatives for the bifurcation curve, we conclude the main theorem.
In Section \ref{S:RO} we apply our result to the reduced Ostrovsky equation, which can be reformulated as a nonlocal equation of the form \eqref{eq:steady} with Fourier symbol $m(k)=k^{-2}$. We recover the well known explicit, even, peaked, periodic traveling-wave given by
\[
	\phi(x)= \frac{2\pi^2-x^2}{18},\qquad \mbox{for}\quad \mu=\frac{\pi^2}{9}
\] 
on $[-\pi,\pi]$ and extended periodically. Moreover, we prove that any periodic traveling-wave $\phi\leq \mu$ is \emph{at least} Lipschitz continuous at its crests; thereby excluding the possibility of periodic, traveling-waves $\phi\leq \mu$ exhibiting a cusp at its crests. Let us mention that the Fourier multiplier $L_2$ for the reduced Ostrovsky equation can be written as a convolution operator, whose kernel can be computed explicitly, see Remark~\ref{rem:ker}. Furthermore, relying on a priori bounds on the wave speed coming from a dynamical system approach for the reduced Ostrovsky equation in \cite{GP}, we are able to obtain a better understanding of the behavior of the global bifurcation branch. 


\bigskip

\section{Functional-analytic setting and general conventions}
\label{S:Setting}

Let us introduce the relevant function spaces for our analysis and fix some notation. We are seeking for $2\pi$-periodic solutions of the steady equation \eqref{eq:steady}. Let us set $\T:=[-\pi,\pi]$, where we identify $-\pi$ with $\pi$. In view of the nonlocal approach via Fourier multipliers, the Besov spaces  on torus $\T$ form a natural scale of spaces to work in. 
We recall the definition and some basic properties of periodic Besov spaces. 

\medskip

Denote by $\mathcal{D}(\T)$ the space of test functions on $\T$, whose dual space, the space of distributions on $\T$, is $\mathcal{D}^\prime(\T)$. If $\mathcal{S}(\Z)$ is the space of rapidly decaying functions from $\Z$ to $\C$ and $\mathcal{S}^\prime(\Z)$ denotes its dual space, let $\mathcal{F}:\mathcal{D}^\prime(\T)\to \mathcal{S}^\prime( \Z)$ be the Fourier transformation on the torus defined by duality on $\mathcal{D}(\T)$ via
 \[
 \mathcal{F}f (k)=\hat f(k):=\frac{1}{2\pi}\int_{\T} f(x)e^{-ixk}\,dx, \qquad f\in \mathcal{D}(\T).
 \]
 
 \medskip
 
 Let $(\varphi)_{j\geq 0}\subset C_c^\infty(\R)$ be a family of smooth, compactly supported functions satisfying
 \[
	 \supp \varphi_0 \subset [-2,2],\qquad \supp \varphi_j \subset [-2^{j+1},-2^{j-1}]\cap [2^{j-1},2^{j+1}] \quad\mbox{ for}\quad j\geq 1,
 \]
 \[
	 \sum_{j\geq 0}\varphi_j(\xi)=1\qquad\mbox{for all}\quad \xi\in\R,
 \]
 and for any $n\in\N$, there exists a constant $c_n>0$ such that 
 \[\sup_{j\geq 0}2^{jn}\|\varphi^{(n)}_j\|_\infty\leq c_n.\]
 For $p,q\in[1,\infty]$ and $s\in\R$, the {periodic Besov spaces} are defined by
 \[
	 B_{p,q}^s(\T):=\left\{ f\in \mathcal{D}^\prime(\T)\mid \|f\|_{B^s_{p,q}}^q:=\sum_{j\geq 0}2^{sjq}\left\|\sum_{k\in \Z} e^{ik(\cdot)} \varphi_j(k)\hat f(k)\right\|_{L^p}^{q}<\infty\right\},
 \]
 with the common modification when $q=\infty$\footnote{One can show that the above definition is independent of the particular choice of $(\varphi)_{j\geq 0}$}. If $s>0$ and $p\in [1,\infty]$, then
  \[
  W^{s,p}(\T)\subset B^s_{p,q}(\T)\subset L^p(\T)\qquad \mbox{for any} \quad q\in [1,\infty].
  \]
 Moreover, for $s>0$, the Besov space $B^s_{\infty,\infty}(\T)$ consisting of functions $f$ satisfying
  \[
  	  \|f\|_{B^s_{\infty,\infty}}=\sup_{j\geq 0}2^{sj}\left\|\sum_{k\in \Z} e^{ik(\cdot)}  \varphi_j(k)\hat f(k)\right\|_\infty < \infty
    \]
   is called {periodic Zygmund space} of order $s$ and we write
  \[
	   \mathcal{C}^s(\T):=B^s_{\infty,\infty}(\T).
	   \]
 
Eventually, for $\alpha \in (0,1)$, we denote by $C^\alpha(\T)$ the space of $\alpha$-H\"older continuous functions on $\T$. 
 If $k\in \N$ and $\alpha\in (0,1)$, then $C^{k,\alpha}(\T)$ denotes the space of $k$-times continuously differentiable functions whose $k$-th derivative is $\alpha$-H\"older continuous on $\T$. To lighten the notation we write $C^s(\T)=C^{\left \lfloor{s}\right \rfloor, s- \left \lfloor{s}\right \rfloor }(\T)$ for $s\geq 0$. 

\medskip

As a consequence of Littlewood--Paley theory, we have the relation $\mathcal{C}^s(\T)=C^s(\T)$ for any $s>0$ with $s\notin \N$; that is, the H\"older spaces on the torus are completely characterized by Fourier series. If $s\in \N$, then $C^s(\T)$ is a proper subset of $\mathcal{C}^s(\T)$ and
\[
	C^1(\T)\subsetneq C^{1-}(\T)\subsetneq \mathcal{C}^1(\T).
\]
 Here, $C^{1-}(\T)$ denotes the space of Lipschitz continuous functions on $\T$. For more details we refer to \cite[Chapter 13]{T3}.
 
 \medskip
 
We are looking for solutions in the class of $2\pi$-periodic, bounded functions with zero mean, the class being denoted by
\[
L^\infty_0(\T):= \{f\in L^\infty(\T) \mid f \mbox{ has zero mean} \}.
\]
  In the sequel we continue to use the subscript $0$ to denote the restriction of a  respective space to its subset of  functions with zero mean. 

\medskip

If $f$ and $g$ are elements in an ordered Banach space, we write $f\lesssim g$ ($f\gtrsim g$) if there exists a constant $c>0$ such that $f\leq c g$ ($f\geq cg$). Moreover, the notation $f\eqsim g$ is used whenever $f\lesssim g$ and $f\gtrsim g$. We denote by $\R_+$ the nonnegative real half axis $\R_+:=[0,\infty]$ and by $\N_0$  the set of natural numbers including zero. The space $\mathcal{L}(X;Y)$ denotes the set of all bounded linear operators from $X$ to $Y$.

\bigskip

\section{Fourier multipliers with homogeneous symbol}
\label{S:Fourier}
 The following result is an analogous statement to the classical Fourier multiplier theorems for nonhomogeneous symbols on Besov spaces (e.g. \cite[Proposition 2.78]{BCD}):

\begin{prop}\label{prop:FM} Let $m>0$ and $\sigma:\R \to \R$ be a function, which is smooth outside the origin and satisfies 
\[
	|\partial^a \sigma(\xi)|\lesssim |\xi|^{-m-a}\qquad \mbox{for all}\quad \xi\neq 0,\quad a\in \N_0.
\]
Then, the Fourier multiplier $L$ defined by \[
Lf=\sum_{k\neq 0}\sigma(k)\hat f(k)e^{ik(\cdot)} 
\] 
belongs to the space ${\mathcal{L}(B^s_{\infty,\infty}}_0(\T);{B^{s+m}_{\infty,\infty}}_0(\T))$.
\end{prop}

 \begin{proof} 
In view of the zero mean property of $f$, the proof can be carried out in a similar form as in \cite[Theorem 2.3 (v)]{AB}, where it is show that a function $f$ belongs to ${B^s_{\infty,\infty}}(\T)$ if and only if
 \[
	 \sum_{k\neq 0}\hat f(k)(ik)^{-m}e^{ik(\cdot)}  \in {B^{s+m}_{\infty,\infty}}(\T).
 \]
 \end{proof}

The above proposition yields in particular that   
\begin{align*}\label{def:L}
L_rf:= \sum_{k\neq 0}|k|^{-r}\hat f(k)e^{ik(\cdot)} , \qquad r>1,
\end{align*}
defines a bounded operator form $\mathcal{C}_0^s(\T)$ to $\mathcal{C}_0^{s+r}(\T)$ for any $s>0$; thereby it is a smoothing operator of order $-r$.
 
 \medskip

We are interested in the existence and regularity properties of solutions of
\begin{equation}\label{Equation}
	-\mu \phi + L_r\phi + \frac{1}{2}\phi-\frac{1}{2}\widehat{\phi^2}(0)=0,\qquad r>1.
\end{equation}
The operator $L_r$ is defined as the inverse Fourier representation 
 \[
 L_r f(x)= \mathcal{F}^{-1}(m_r\hat f)(x),
 \]
where $m_r(k)=|k|^{-r}$ for $k\neq 0$ and $m_r(0)=0$.  In view of the convolution theorem, we define the integral kernel
\begin{align}\label{def:K}
K_r(x):= 2\sum_{k=1}^\infty |k|^{-r}\cos\left(xk\right), \qquad x\in \T,
\end{align}
so that the action of $L_r$  is described by the convolution
 \begin{equation*}\label{eq:convolution_kernel}
	 L_rf=K_r*f.
 \end{equation*}
One can then express  equation \eqref{Equation} as
\begin{equation*}\label{eq:RO}
-\mu \phi +K_r*\phi -\frac{1}{2}\widehat{\phi^2}(0)=0, \qquad K_r:= \mathcal{F}^{-1}( m_r).
\end{equation*}

\medskip

In what follows we examine the kernel $K_r$. We start by recalling some general theory on {completely monotonic} sequences taken from \cite{Guo, Widder}.
\begin{defi}
A sequence $(\mu_k)_{k\in\N_0}$ of real numbers is called \emph{completely monotonic} if its elements are nonnegative and
\[
	(-1)^n\Delta^n\mu_k \geq 0\qquad \mbox{for any}\quad n,k\in\N_0,
\]
where $\Delta^0\mu_k=\mu_k$ and $\Delta^{n+1}\mu_k=\Delta^n\mu_{k+1}-\Delta^n \mu_k$.
\end{defi}
\begin{defi}
	A function $f:[0,\infty)\to\R$ is called \emph{completely monotone} if it is continuous on $[0,\infty)$, smooth on the open set $(0,\infty)$, and satisfies
	\[
	(-1)^n f^{(n)}(x)\geq 0\qquad\mbox{for any}\quad x>0.
	\]
\end{defi}
For completely monotonic sequences we have the following theorem, which can be considered as the discrete analog of Bernstein's theorem on completely monotonic functions.
\begin{theo}[\cite{Widder}, Theorem 4a]\label{thm:B}
	A sequence $(\mu_k)_{k\in\N_0}$ of real numbers is  completely monotonic if and only if
	\[
		\mu_k=\int_0^1 t^k d\sigma(t),
	\]
	where $\sigma$ is nondecreasing and bounded for $t\in[0,1]$.
\end{theo}
There exists a close relationship between completely monotonic sequences and completely monotonic functions.
\begin{lem}[\cite{Guo}, Theorem 5]\label{lem:CM}
	Suppose that $f:[0,\infty)\to \R$ is completely monotone, then for any $a\geq 0$ the sequence $(f(an))_{n\in\N_0}$ is completely monotonic.
\end{lem}

We are going to use the theory on completely monotonic sequences to prove the following theorem, which summarizes some properties of the kernel $K_r$.

\begin{theo}[Properties of $K_r$]\label{thm:P}
	Let $r>1$. The kernel $K_r$ defined in \eqref{def:K} has the following properties:
	\begin{itemize}
		\item[a)] $K_r$ is even, continuous, and has zero mean.
		\item[b)] $K_r$ is smooth on $\T\setminus\{0\}$ and decreasing on $(0,\pi)$.
		\item[c)]  $K_r \in W^{r-\e,1}(\T)$ for any $\e \in (0,1)$. In particular, $K_r^\prime$ is integrable and $K_r$ is $\alpha$-H\"older continuous with $\alpha \in (0,r-1)$  if $r\in (1,2]$, and continuously differentiable if $r> 2$.
		\end{itemize} 
	
\end{theo}

\begin{proof}
	Claim a) follows directly form the definition of $K_r$ and $r>1$. 
	Now we want to prove part b).  Set 
	\[
		\mu_k:=(k+1)^{-r}\qquad \mbox{for}\quad k\in\N_0.
	\]
 Clearly $x\mapsto (x+1)^{-r}$ is completely monotone on $(0,\infty)$. Thus, Lemma \ref{lem:CM} guarantees that $(\mu_k)_{k\in \N_0}$ is a completely monotonic sequence. By Theorem \ref{thm:B}, there exists a nondecreasing and bounded function $\sigma_r:[0,1]\to\R$ such that
	\[
		(k+1)^{-r}=\int_0^1 t^{k}\, d\sigma_r(t)\qquad\mbox{for any}\quad k\geq 0.
	\]  
	In particular
	\[
		|k|^{-r}=\int_0^1 t^{|k|-1}\, d\sigma_r(t)\qquad\mbox{for any}\quad k\neq 0.
	\]
	The coefficients $t^{|k|-1}$ can be written as
	\[
		t^{|k|-1}=\int_\T f(t,x)e^{-ixk}\,dx\qquad \mbox{for}\quad k\neq 0,
	\]
	where
	\[
	f(t,x)=\sum_{k\neq 0}t^{|k|-1}e^{ixk}+a_0(t)
	\]
	for some bounded function $a_0:(0,1)\to \R$.
	Thereby,
	\begin{align*}
		|k|^{-r}= \int_\T \int_0^1 f(t,x)\,d\sigma_r(t)e^{-ixk}\,dx\qquad\mbox{for any}\quad k\neq 0.
	\end{align*}
	In particular, we deduce that
	\[
		\int_0^1 f(t,x)\,d\sigma_r(t)=\sum_{k\neq 0}	|k|^{-r}e^{ixk}=K_r(x).
	\]
	Notice that we can compute $f$ explicitly as
	\begin{align*}
	 f(t,x)-a_0(t)&= \sum_{k\neq 0}t^{|k|-1}e^{ixk}=2\sum_{k=1}^\infty t^{k-1}\cos(xk)=2\sum_{k=0}t^{k}\cos(x(k+1))\\
	 &=2\re \left(e^{ix}\sum_{k=0}t^{k}e^{ixk}\right)=2\re \left(e^{ix}\sum_{k=0}^\infty \left(te^{ix}\right)^k \right).
	\end{align*}
	Thus, for $x\in (0,\pi)$, we have that
	\[
		f(t,x)=2\re\left(e^{ix}\frac{1}{1-te^{ix}}\right)+a_0(t)=\frac{2(\cos(x)-t)}{1-t^2\cos(x)+t^4}+a_0(t).
	\]
 Consequently, on the interval $(0,\pi)$, the kernel $K_r$ is represented by
	\begin{equation*}\label{eq:rep}
		K_r(x)=\int_0^1 \left(\frac{2(\cos(x)-t)}{1-t\cos(x)+t^2}+a_0(t)\right)\,d\sigma_r(t).
	\end{equation*}
	From here it is easy to deduce that $K_r$ is smooth on $\T\setminus\{0\}$ and decreasing on $(0,\pi)$, which completes the proof of b).
	 Regarding the regularity of $K_r$ claimed in c), let $\e\in (0,1)$ be arbitrary.
		On the subset of zero mean functions of $W^{r-\e,1}(\T)$ an equivalent norm is given by
		\[
		\|K_r\|_{W^{r-\e,1}_0}\eqsim \|\mathcal{F}^{-1}\left(|\cdot|^{r-\e}\hat K_r\right)\|_{L^1}.
		\]
		Thereby, $K_r$ is in  $W^{r-\e,1}_0(\T)$ if and only if the function
		\[
		x\mapsto \mathcal{F}^{-1}(|\cdot|^{r-\e}\hat K_r)(x)= 2\sum_{k=1}^\infty |k|^{r-\e-r}\cos(xk)=2\sum_{k=1}^\infty |k|^{-\e}\cos(xk)
		\]
		is integrable over $\T$. Now, this follows by a classical theorem on the integrability of trigonometric transformations (cf. \cite[Theorem 2]{Boas} ), and we deduce the claimed regularity and integrability of $K_r^\prime$. The continuity properties are a direct consequence of Sobolev embedding theorems, see \cite[Theorem 4.57]{Demengel}.
\end{proof}
	\begin{rem}\label{rem:ker}\emph{
		The proof of Theorem \ref{thm:P} includes a general approach on the relation between the symbol and the monotonicity property of the corresponding Fourier multiplier. However, there exists even a more explicit expression of $K_r$ in terms of the Gamma function (cf. \cite[Section 5.4.3]{PBM}) given by
			\[
			K_r(x)=\frac{2}{ \Gamma(r)}\int_0^\infty \frac{t^{r-1}(e^t\cos\left(x\right)-1)}{1-2e^t\cos\left(x\right)+e^{2t}}\,dt,\qquad x\in [0,\pi],\qquad r>1.
			\]
			Moreover, we would like to point out that if $r=2n$, $n\in\N$, we have that
			\[
				K_{2n}(x)=\frac{(-1)^{n-1}}{(2n)!}(2\pi)^{2n}B_{2n}\left(\frac{x}{2\pi}\right),\qquad x\in [0,\pi],
			\]
			where $B_{m}$ is the $m$-th Bernoulli polynomial. If $r=2$ (which corresponds to the case of the reduced Ostrovsky equation), or $r=4$, then $B_2(x)=x^2-x+\frac{1}{6}$, $B_4(x)=x^4-2x^3+x^2-\frac{1}{30}$, and
			\[
				K_2(x)=\frac{1}{2}\left(|x|-\pi \right)^2-\frac{\pi^2}{6},\qquad K_4(x)= \frac{\pi^4}{45}-\frac{1}{24}\left(x^2-2\pi|x|\right)^2,\qquad  x\in [-\pi,\pi].
			\]
			\begin{center}
				\begin{figure}[h]
				\begin{tikzpicture}[scale=0.8]
				\scriptsize
				\draw[->](0,-2)--(0,4);
				\draw[->](-4,0)--(4,0);
				\draw[-](-pi,0.1)--(-pi,-0.1) node[below] {$-\pi$};
				\draw[-](pi,0.1)--(pi,-0.1) node[below] {$\pi$};
				\draw[domain=-pi:pi,smooth, thick, variable=\x,luh-dark-blue] plot ({\x},{0.5*(abs(\x)-pi)^2-pi^2/6});
					\draw[domain=-pi:pi,smooth,thick, variable=\x,luh-dark-blue, dashed] plot ({\x},{(pi^4)/45-(1/24)*(\x*\x -2*pi*abs(\x))^2});
				\end{tikzpicture}
				\caption{Plot of $K_2$ (solid) and $K_4$ (dashed).}
								\end{figure}
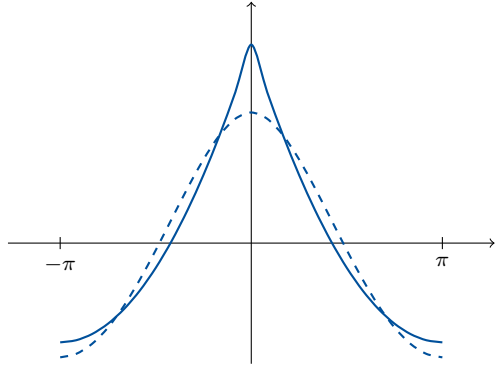
			\end{center}
		}

	\end{rem}


\begin{lem}\label{lem:touch} Let $r>1$.  The operator $L_r$ is parity preserving on $L^\infty_0(\T)$. Moreover, if $f,g \in L^\infty_0(\T)$ are odd functions satisfying $f(x)\geq g(x)$ on $[0,\pi]$, then either
\[
	L_rf(x)> L_rg(x)\qquad \mbox{for all}\quad x\in  (0,\pi),
\]
or $f=g$ on $\T$.
\end{lem}

\begin{proof} The fact that $L_r$ is parity preserving is an immediate consequence of the evenness of the convolution kernel.  In order to prove the second assertion, assume that $f,g\in L^\infty_0(\T)$ are odd, satisfying $f(x)\geq g(x)$ on $[0,\pi]$ and that there exists $x_0\in (0,\pi)$ such that $f(x_0)=g(x_0)$.  Using the zero mean property of $f$ and $g$,  we obtain that
\[
	L_rf(x_0)-L_rg(x_0)=\int_{-\pi}^\pi (K_r(x_0-y)-\min K_r)\left( f(y)-g(y)\right)\,dy>0,
\]
where $\min K_r$ denotes the minimum of $K_r$ on $\T$.
In view of $K_r$ being nonconstant and $K_r(y)-\min K_r\geq 0$ for all $y\in \T$, we conclude that 
\[
		L_rf(x_0)-L_rg(x_0)>0,
\]
which is a contradiction unless $f=g$ on $\T$.
\end{proof}

 \bigskip

\section{A priori properties of periodic traveling-wave solutions}
\label{S:Properties}
In the sequel, let $r>1$ be fixed. We consider $2\pi$-periodic solutions of
\begin{equation}
\label{eq:gRO}
-\mu\phi +L_r\phi+\frac{1}{2}\phi^2-\frac{1}{2}\widehat{\phi^2}(0)=0.
\end{equation}

The existence of solutions is subject of Section \ref{S:Global}, where we use analytic bifurcation theory to first construct small amplitude solutions and then extend this bifurcation curve to a global continuum terminating in a highest, traveling wave. Aim of this section is to provide a priori properties of traveling-wave solutions $\phi\leq\mu$. In particular, we show  that any nontrivial, even solution $\phi \leq \mu$, which is nondecreasing on the half period $(-\pi,0)$ and  attaining its maximum at $\phi(0)=\mu$ is precisely Lipschitz continuous. This holds true for any $r>1$, see Theorem \ref{thm:reg}.

\medskip

\emph{
	We would like to point out that the subsequent analysis can be carried out in the very same manner for $2P$-periodic solutions, where $P\in (0,\infty)$ is the length of a finite half period. }

\medskip

Let us start with a short observation.
\begin{lem} \label{lem:distance}
If $\phi \in C_0 (\T)$ is a nontrivial solution of \eqref{eq:gRO}, then
\[
	\phi(x_M)+ \phi(x_m)\geq 2 \left( \mu -\|K_r\|_{L^1}\right),
\]
where $\phi(x_M)=\max_{x\in\T}\phi(x)$ and $\phi(x_m)=\min_{x\in\T}\phi(x)$.
\end{lem}
\begin{proof} If $\phi\in C_0(\T)$ is a nontrivial solution of \eqref{eq:gRO}, then $\phi(x_M)>0>\phi(x_m)$ and
\begin{align*}
	\mu (\phi(x_M)-\phi(x_m))&=K_r*\phi(x_M)-K_r*\phi(x_m)+\frac{1}{2}\left( \phi^2(x_M)- \phi^2 (x_m) \right)\\
	& \leq \|K_r\|_{L^1}(\phi(x_M)-\phi(x_m)) + \frac{1}{2}\left( \phi(x_M) - \phi(x_m)\right)\left( \phi(x_M) + \phi(x_m)\right),
\end{align*}
which proves the statement.
\end{proof}

In what follows it is going to be convenient to write \eqref{eq:gRO} as
\begin{equation}\label{trav:eqn}
\frac{1}{2}(\mu - \phi)^{2} = \frac{1}{2}\mu^{2} - L_r\phi+\frac{1}{2}\widehat{\phi^2}(0).
\end{equation}

In the next two lemmata we establish a priori properties of periodic solutions of \eqref{trav:eqn} requiring solely boundedness.

\begin{lem}\label{lem:c1}
Let $\phi \in  L_0^\infty(\T)$ be a solution of \eqref{trav:eqn}, then $\left(\mu-\phi \right)^2 \in C^1(\T)$ and
\[
	\left\|\frac{d}{dx}\left(\mu-\phi \right)^2\right\|_\infty \leq 2\left\|K_r^\prime\right\|_{L^1}\|\phi\|_\infty \quad \mbox{for all}\quad x \in \T.
\]
\end{lem}
\begin{proof}
We can read of from \eqref{trav:eqn}, that the derivative of $(\mu-\phi)^2$ is given by
\[
	\frac{d}{dx}(\mu-\phi)^2(x)=-2 K_r^\prime*\phi(x).
\]
Since $K_r^\prime$ and $\phi$ are integrable over $\T$ (cf. Theorem \ref{thm:P}), the convolution on the right hand side is continuous and the claimed estimate follows.
\end{proof}

\begin{lem}\label{lem:uniform_bound}
Let $\phi\in  L_0^\infty(\T)$ be a solution of \eqref{trav:eqn}, then
\[
	\|\phi\|_\infty \leq 2 \left(\mu + \|K_r\|_{L^1}\right) +2\pi\|K_r^\prime\|_{L^1}.
\]
\end{lem}

\begin{proof} If $\phi=0$, there is nothing to prove. Therefore it is enough to assume that $\phi$ is a nontrivial solution. From Lemma \ref{lem:c1} we know that $(\mu-\phi)^2$ is a continuously differentiable function. In view of $\phi$ being a function of zero mean and $(\mu-\phi)^2$ being continuous, we deduce the existence of $x_0\in \T$ such that
\[
	(\mu-\phi)^2(x_0)=\mu^2.
\]
By the mean value theorem, we obtain that
\[
	(\mu-\phi)^2(x)= \left[(\mu - \phi)^2\right]^\prime(\xi)(x-x_0)+\mu^2
\]
for some $\xi \in \T$ and
\begin{align*}
\widehat{\phi^2} (0)&= \frac{1}{2\pi}\int_{-\pi}^\pi \phi^2(x)\,dx =\frac{1}{2\pi}\int_{-\pi}^\pi\left[(\mu - \phi)^2\right]^\prime(\xi)(x-x_0)\,dx,
\end{align*}
where we used that $\phi$ has zero mean. Again by Lemma \ref{lem:c1} we can estimate the term above generously by
\[
\widehat{\phi^2} (0) \leq 2\pi\|K_r^\prime\|_{L^1}\|\phi\|_\infty.
\]
Using that $\phi$ solves \eqref{eq:gRO}, we obtain
\begin{align*}
\|\phi\|^2_\infty \leq 2 (\mu + \|K_r\|_{L^1})\|\phi\|_\infty +  2\pi\|K_r^\prime\|_{L^1}\|\phi\|_\infty.
\end{align*}
Dividing by $\|\phi\|_\infty$ yields the statement.


\end{proof}

From now on we restrict our considerations on  periodic solutions of \eqref{trav:eqn}, which are even and nondecreasing on the half period $[-\pi,0]$. 

\begin{lem}\label{lem:nod} Any  nontrivial, even solution $\phi \in C_0^1(\T)$ of \eqref{trav:eqn} which is nondecreasing on $(-\pi,0)$ satisfies
\[
	\phi^\prime(x)>0 \qquad \mbox{and}\qquad   \phi(x)<\mu \qquad \mbox{on}\quad (-\pi,0).
\]
Moreover, if $\phi\in C_0^2(\T)$, then $\phi^{\prime \prime}(0)<0$.
\end{lem}

\begin{proof}
Assuming that $\phi\in C_0^1(\T)$ we can take the derivative of \eqref{trav:eqn} and obtain that
\begin{equation*}
(\mu-\phi)\phi^\prime(x)=L_r\phi^\prime (x).
\end{equation*}
Due to the assumption that $\phi^\prime\geq 0$ on $(-\pi,0)$ it is sufficient to show that 
\begin{equation}\label{eq:ineq1}
	L_r\phi^\prime(x)> 0 \qquad \mbox{on}\quad (-\pi,0)
\end{equation}
to prove the statement. In view of $\phi^\prime$ being odd with $\phi^\prime(x)\geq 0$ on $[-\pi,0]$,  the desired inequality \eqref{eq:ineq1} follows from Lemma \ref{lem:touch}.
In order to prove the second statement, let us assume that $\phi\in C_0^2(\T)$. Differentiating \eqref{trav:eqn} twice yields
\begin{equation*}
(\mu-\phi)\phi^{\prime\prime}(x)=L_r\phi^{\prime\prime} (x)+(\phi^\prime)^2(x).
\end{equation*}
In particular, we have that
\[
(\mu-\phi)\phi^{\prime\prime}(0)=L_r\phi^{\prime\prime} (0).
\]
We are going to show that $L_r\phi^{\prime\prime} (0)<0$, which then (together with the first part) proves the statement. Using the evenness of $K_r$ and $\phi^{\prime\prime}$, we compute
\begin{align*}
\frac{1}{2}L_r\phi^{\prime\prime} (0) &= \frac{1}{2}\int_{-\pi}^\pi K_r(y)\phi^{\prime\prime}(y)\,dy \\
&= \int_{0}^\pi K_r(y)\phi^{\prime\prime}(y)\,dy \\
&=  \int_{0}^\e K_r(y)\phi^{\prime\prime}(y)\,dy+\int_{\e}^\pi K_r(y)\phi^{\prime\prime}(y)\,dy\\
&=\int_{0}^\e K_r(y)\phi^{\prime\prime}(y)\,dy + K_r(\e)\phi^\prime(\e)- \int_\e^\pi K_r^\prime(y)\phi^\prime(y)\,dy.
\end{align*}
Notice that the first integral on the right hand side tends to zero if $\e\to 0$, so does the second term in view of $\phi$ being differentiable and $K_r$ continuous on $\T$. Concerning the last integral, we observe that 
\[
\frac{1}{2}L_r\phi^{\prime\prime} (0)=- \lim_{\e \to 0^+}\int_\e^\pi K_r^\prime(y)\phi^\prime(y)\,dy<0,
\]
since $K_r^\prime$ and $\phi^\prime$ are negative on $(0,\pi)$.
\end{proof}

 We continue by showing that any  bounded solution $\phi$ of \eqref{trav:eqn} that satisfies $\phi<\mu$ is smooth.

\begin{theo}\label{th:phi:prop}
Let $\phi\le \mu$ be a bounded solution of \eqref{trav:eqn}. Then:
\begin{itemize}
\item[(i)] If $\phi<\mu$ uniformly on $\T$, then $\phi\in C^{\infty}(\R)$.
\item[(ii)] Considering $\phi$ as a periodic function on $\R$ it is smooth on any open set where $\phi<\mu$. 
\end{itemize}
\end{theo}
\begin{proof}
 Let $\phi<\mu$ uniformly on $\T$. Recalling Proposition \ref{prop:FM}, we know that the operator $L_r$ maps ${B^s_{\infty,\infty}}_0(\T)$ into ${B^{s+r}_{\infty,\infty}}_0(\T)$ for any $s\in\R$.
Moreover, if $s>0$ then the Nemytskii operator
\begin{align*}
f\mapsto \mu - \sqrt{\frac{1}{2}\mu^{2}- f}
\end{align*}
maps $  {B^s_{\infty,\infty}}_0(\T)$ into itself for $f<\frac{1}{2}\mu^2$. From \eqref{trav:eqn} we see that for any solution $\phi<\mu$ we have 
\[L\phi-\frac{1}{2}\widehat \phi^2(0)<\frac{1}{2}\mu^2.\] 
Thus,
\begin{align}\label{maps:reg}
\begin{split}
\left[ L_r\phi\mapsto \sqrt{\frac{1}{2}\mu^2- L\phi+\frac{1}{2}\hat \phi^2(0)}\right]\circ \left[ \phi\mapsto L_r\phi-\frac{1}{2}\widehat \phi^2(0)\right]:  {B^s_{\infty,\infty}}_0(\T)
\to   {B^{s+r}_{\infty,\infty}}_0(\T),
\end{split}
\end{align}
for all $s\geq 0$. Eventually, \eqref{trav:eqn} gives rise to
\begin{align*}
\phi = \mu-\sqrt{\mu^{2} - 2L_r\phi+ \hat{\phi}^{2}(0) }.
\end{align*}
Hence, an iteration argument in $s$ guarantees that $\phi\in C^{\infty}(\T)$. In order to prove the statement on the real line, recall that any Fourier multiplier commutes with the translation operator. Thus, if $\phi$ is a periodic solution of \eqref{trav:eqn}, then so is $\phi_h:=\phi(\cdot +h)$ for any $h\in \R$. The previous argument implies that $\phi_h \in C^\infty(\T)$ for any $h\in\R$, which proves statement (i).
In order to prove part (ii) let $U\subset \R$ be an open subset of $\R$ on which $\phi <\mu$. Then, we can find an open cover $U=\cup_{i\in I}U_i$, where for any $i\in I$ we have that $U_i$ is connected and satisfies $|U_i|<2\pi$. Due to the translation invariance of \eqref{trav:eqn} and part (i), we obtain that $\phi$ is smooth on $U_i$ for any $i\in I$. Since $U$ is the union of open sets, the assertion follows.
\end{proof}

\begin{theo}\label{thm:regularity}
Let $\phi\leq \mu$ be an even solution of \eqref{trav:eqn}, which is nondecreasing on $[-\pi, 0]$. If $\phi$ attains its maximum at $\phi(0)=\mu$, then $\phi$ cannot belong to the class $C^1(\T)$.
\end{theo}

\begin{proof} 
Assuming that $\phi \in C^1(\T)$, the same argument as in Lemma \ref{lem:c1} implies that the function $(\mu-\phi)^2$ is twice continuously differentiable and its Taylor expansion in a neighborhood of $x=0$ is given by
\begin{align}\label{eq:taylor:ser}
	(\mu-\phi)^2(x)=[(\mu-\phi)^2]^{\prime}(0)x+\frac{1}{2}[(\mu-\phi)^2]^{\prime\prime}(\xi)x^2
\end{align}
for some $|\xi|\in (0,|x|)$ where $|x|\ll 1$.
Since $\phi$ attains a local maximum at $x=0$, its first derivative above vanishes at the origin whereas the second derivative is given by
\[
	\frac{1}{2}[(\mu-\phi)^2]^{\prime\prime}(\xi)=-K_r^{\prime}*\phi^\prime(\xi).
\]
We aim to show that in a small neighborhood of zero the right hand side is strictly bounded away from zero. Set 
$f(\xi):=-K_r^\prime*\phi^\prime(\xi)$. 
Using that $K_r$ and $\phi$ are even functions with $K_r^\prime$ and $\phi^\prime$ being negative on $(0,\pi)$, we find that
\[
	f(0)=-K_r^\prime*\phi^\prime(0)=2\int_0^\pi K_r^\prime(y)\phi^\prime(y)\,dy=c>0
\]
for some constant $c>0$. Since $f$ is even (cf. Lemma \ref{lem:touch}) and continuous,  there exists $|x_0|\ll 1$ and a constant $c_0>0$ such that 
\[
\frac{1}{2}[(\mu-\phi)^2]^{\prime\prime}(\xi)=f(\xi)\geq c_0,\qquad \mbox{for all}\quad \xi\in (0,|x_0|).
\]
 Thus, considering the Taylor series~\eqref{eq:taylor:ser} in a neighborhood of zero, we have that
\[
	(\mu-\phi)^2(x)\gtrsim x^2\qquad \mbox{for}\quad |x|\ll 1,
\]
which in particular implies that
\begin{equation*}\label{eq:Lip}
 \frac{\mu-\phi(x)}{|x|}\gtrsim 1\qquad \mbox{for}\quad |x|\ll 1.
\end{equation*}
Passing to the limit $x\to 0$, we obtain a contradiction to  $\phi^\prime(0)=0$.
\end{proof}

 We are now investigating the precise regularity of a solution $\phi$, which attains its maximum at $\phi(0)=\mu$. 

\begin{theo}
	\label{thm:reg}
Let $\phi\leq \mu$ be an even solution of \eqref{trav:eqn}, which is nondecreasing on $[-\pi, 0]$. If $\phi$ attains its maximum at $\phi(0)=\mu$, then the following holds:
\begin{itemize}
\item[(i)] $\phi\in C^\infty(\T \setminus \{0\})$ and $\phi$ is strictly increasing on $(-\pi,0)$.
 \item[(ii)] $\phi\in C_0^{1-}(\T)$, that is $\phi$ is Lipschitz continuous.
\item[(iii)] $\phi$ is precisely Lipschitz continuous at $x=0$, that is
\begin{align}\label{itm:3}
\mu - \phi(x) \simeq |x| \qquad \mbox{for}\; |x|\ll 1.
\end{align}
\end{itemize}
\end{theo}

\begin{proof}
\begin{itemize} 
\item[(i)]  Assume that $\phi\leq \mu$ is a solution which is even and nondecreasing on $(-\pi,0)$.
Let $x\in (-\pi,0)$ and $h\in (0,\pi)$. Notice that by periodicity and evenness of $\phi$ and the kernel $K_r$, we have that
\begin{align*}
K_r*\phi(x+h)&-K_r*\phi(x-h)\\
&=\int_{-\pi}^{0} \left(K_r(x-y)-K_r(x+y)\right)\left(\phi(y+h)-\phi(y-h)\right)\,dy.
\end{align*}
The integrand is nonnegative, since $K_r(x-y)-K_r(x+y)> 0$ for $x,y\in (-\pi,0)$ and $\phi(y+h)-\phi(y-h)\geq 0$ for $y\in (-\pi,0)$ and $h\in (0,\pi)$ by assumption that $\phi$ is even and nondecreasing on $(-\pi,0)$. Since $\phi$ is  a nontrivial solution and $K_r$ is not constant, we deduce that
\begin{equation}\label{eq:mon}
K_r*\phi(x+h)-K_r*\phi(x-h)>0
\end{equation}
for any $h\in (0,\pi)$. 
Moreover, we have that
\[
\frac{1}{2}\left( 2\mu-\phi(x)-\phi(y) \right)\left(\phi(y)-\phi(x)\right)=K_r*\phi(x)-K_r*\phi(y)
\]
for any $x,y\in \T$.
 Hence $K_r*\phi(x)=K_r*\phi(y)$ if and only if $\phi(x)=\phi(y)$. 
 In view of \eqref{eq:mon}, we obtain that
\[
\phi(x+h)\neq \phi(x-h)\qquad\mbox{for any}\quad h\in (0,\pi).
\]
Thereby, $\phi$ is strictly increasing on $(-\pi,0)$. In view of Theorem \ref{th:phi:prop}, $\phi$ is smooth on $\T\setminus\{0\}$.

\item[(ii)] 
In order to prove the Lipschitz regularity at the crest, we make use of a simple \emph{bootstrap argument}. We would like to emphasize that the following argument strongly relies on the fact that we are dealing with a smoothing operator of order $-r$, where $r>1$. Let us assume that $\phi$ is  \emph{not} Lipschitz continuous and prove a contradiction. If $\phi\leq \mu$ is merely a bounded function, the regularization property of $L_r$ implies immediately the $\phi$ is a priori $\frac{1}{2}$-H\"older continuous. To see this, recall that 
\[
\frac{1}{2}\left( 2\mu-\phi(x)-\phi(y) \right)\left(\phi(y)-\phi(x)\right)=L_r\phi(x)-L_r\phi(y).
\]
Using $\phi\leq \mu$, we deduce that
\[
\frac{1}{2}\left(\phi(x)-\phi(y)\right)^2\leq |L_r\phi(x)-L_r\phi(y)|.
\]
Since $L_r:L_0^\infty(\T)\to\mathcal{C}_0^r(\T)$, where $r>1$, the right hand side can be estimated by a constant multiple of $|x-y|$. An immediate consequence is the $\frac{1}{2}$-H\"older continuity of $\phi$. Since $\phi$ is smooth in $\T\setminus\{0\}$, we can differentiate the equality
\[
	\frac{1}{2}(\mu-\phi)^2(x)=K_r*\phi(0)-K_r*\phi(x)
\]
for $x\in (-\pi,0)$ and obtain that
\begin{equation*}\label{eq:BA}
	(\mu-\phi)\phi^\prime(x)=\left(K_r*\phi\right)^\prime(x)-\left(K_r*\phi\right)^\prime(0),
\end{equation*}
where we are using that $\left(K_r*\phi\right)^\prime(0)=0$.
If $\phi$ is $\frac{1}{2}$-H\"older continuous, then $K_r*\phi\in \mathcal{C}_0^{\frac{1}{2}+r}(\T)$. In view of $r>1$, we gain at least some H\"older regularity for  $(K_r*\phi)^\prime$. 
Thereby,
\begin{equation}\label{eq:BE}
(\mu-\phi)\phi^\prime(x)\lesssim |x|^{a}
\end{equation}
for some $a\in (\frac{1}{2},1]$. By assumption that $\phi$ is not Lipschitz continuous at $x=0$, the above estimate guarantees that $\phi$ is at least $a$-H\"older continuous, where $a>\frac{1}{2}$. We aim to bootstrap this argument to obtain Lipschitz regularity of $\phi$ at $x=0$. If above $\frac{1}{2}+r>2$, we use that  $K_r*\phi\in \mathcal{C}_0^{\frac{1}{2}+r}(\T)\subset C_0^2(T)$, which guarantees that its derivative is at least Lipschitz continuous  ($a=1$ in \eqref{eq:BE}) and we are done. If $\frac{1}{2}+r\leq 2$, we merely obtain an improved $a$-H\"older regularity of $\phi$. However, repeating the argument finitely many times, yields that $\phi$ is indeed Lipschitz continuous at $x=0$, that is
\begin{equation}\label{eq:upper}
\mu-\phi(x)\lesssim |x|,\qquad\mbox{for}\quad |x|\ll 1.
\end{equation}

\item[(iii)] In view of the upper bound \eqref{eq:upper} we are left to establish an according lower bound for $|x|\ll 1$ to prove the claim \eqref{itm:3}. To achieve this, we show that the derivative is positive and bounded away from zero on $(-\pi,0)$. Let $\xi \in (-\pi,0)$, then
\begin{align*}
	(\mu-\phi)\phi^\prime(\xi)=K_r^\prime * \phi(\xi)
	=\int_{-\pi}^0 \left(K_r(\xi-y)-K_r(\xi+y) \right)\phi^\prime(y)\,dy.
\end{align*}
Using the upper bound established in \eqref{eq:upper}, we divide the above equation by $(\mu-\phi)(\xi)$ and obtain that

\begin{align*}
\phi^\prime(\xi)\gtrsim \int_{-\pi}^0 \frac{K_r(\xi-y)-K_r(\xi+y)}{|\xi|} \phi^\prime(y)\,dy.
\end{align*}
Our aim is to show that $\liminf_{\xi \to -0}\phi^\prime(\xi)$ is strictly bounded away from zero. 
We have that 
\begin{align*}
	\lim_{\xi \to -0}&\frac{K_r(\xi-y)-K_r(\xi+y)}{|\xi|}\\&= \lim_{\xi \to -0}\left(\frac{K_r(y-\xi)-K_r(y)}{\xi}+\frac{K_r(y)-K_r(y+\xi)}{\xi}\right)\frac{\xi}{|\xi|}=2K_r^\prime(y) 
\end{align*}
for any $y\in (-\pi,0)$ (keep in mind that $\xi<0$).
The integrability of $K_r^\prime$ allows us to estimate
\begin{equation}\label{eq:upperbound}
	\liminf_{\xi \to 0}\phi^\prime(\xi)\gtrsim 2\int_{-\pi}^{0} K_r^\prime(y)\phi^\prime(y)\, dy =c
\end{equation}
for some constant $c>0$, since $\phi$ as well as $K_r$ are strictly increasing on $(-\pi,0)$.
Let $x<0$ with $|x|\ll 1$.
Applying the mean value theorem for $\phi$ on the interval $(x,0)$ yields that
\begin{align*}
\frac{\phi(0) - \phi(x)}{|x|} =  \phi'(\xi) \qquad  \text{for some}\quad |\xi|\ll 1.
\end{align*} 
In accordance with \eqref{eq:upperbound}, we conclude that

\[
	\mu - \phi (x) \simeq |x|\qquad \mbox{for} \; |x|\ll 1.
\]
\end{itemize}
\end{proof}

\begin{rem} \emph{The above theorem implies in particular that \emph{any} periodic solution $\phi \leq \mu$ of \eqref{trav:eqn}, which is monotone on a half period, is at least Lipschitz continuous. Thereby,  the existence of corresponding cusped traveling-wave solutions satisfying $\phi \leq \mu$ is a priori excluded.}
\end{rem}

\begin{lem}\label{lem:lowerbound:aux}
	Let $\phi \leq \mu$ be an even solution of \eqref{trav:eqn}, which is nondecreasing on $[-\pi,0]$. Then there exists a constant $\lambda=\lambda(r)>0$, depending only on the kernel $K_r$, such that
	\[
	\mu-\phi(\pi)\geq \lambda \pi.
	\]
\end{lem}
\begin{proof}
	Let us pick $x\in [-\frac{3}{4}\pi,-\frac{1}{4}\pi]$.
	Then,
	\begin{align}\label{eq:estimateA}
	\begin{split}
	(\mu-\phi(\pi))\phi^\prime(x)&\geq (\mu-\phi(x))\phi^\prime(x)\\
	&=\int_{-\pi}^{0}\left(K_r(x-y)-K_r(x+y)\right)\phi^\prime(y)\,dy\\
	&\geq \int_{-\frac{3}{4}\pi}^{-\frac{1}{4}\pi}\left(K_r(x-y)-K_r(x+y)\right)\phi^\prime(y)\,dy,
	\end{split}
	\end{align}
	using the evenness of the kernel $K_r$, implying that $K_r(x-y)-K_r(x+y)>0$ for $x,y\in (-\pi,0)$. We observe that there exists a constant $\lambda=\lambda(r)>0$, depending only on the kernel $K_r$, such that
	\[
	K_r(x-y)-K_r(x+y)\geq 2\lambda \qquad \mbox{for all}\quad x,y \in \left(-\frac{3}{4}\pi,-\frac{1}{4}\pi\right).
	\]
	Thus, integrating \eqref{eq:estimateA} with respect to $x$ over $\left(-\frac{3}{4}\pi,-\frac{1}{4}\pi\right)$ yields
	\begin{align*}
	(\mu-\phi(\pi))\left(\phi\left(-\frac{1}{4}\pi\right)-\phi\left(-\frac{3}{4}\pi\right)\right)&\geq\int_{-\frac{3}{4}\pi}^{-\frac{1}{4}\pi}\left(\int_{-\frac{3}{4}\pi}^{-\frac{1}{4}\pi} K_r(x-y)-K_r(x+y)\, dx\right)\phi^\prime(y)\,dy\\
	&\geq \lambda \pi \left(\phi\left(-\frac{1}{4}\pi\right)-\phi\left(-\frac{3}{4}\pi\right)\right).
	\end{align*}
	In view of $\phi$ being strictly increasing on $(-\pi,0)$ (cf. Theorem \ref{thm:reg} (i)), we can divide the above inequality by the positive number $\left(\phi\left(-\frac{1}{4}\pi\right)-\phi\left(-\frac{3}{4}\pi\right)\right)$ and thereby affirm the claim.
\end{proof}

We close this section by proving that there is a natural bound on $\mu$ above which there do not exist any nontrivial, continuous solutions, which satisfying the uniform bound $\phi \leq \mu$. This is going to be used to exclude certain alternatives in the analysis of the global bifurcation curve in Section \ref{S:Global}.

\begin{lem}\label{lem:bound_mu}
	If $\mu\geq 2 \|K_r\|_{L^1}$, then there exist no nontrivial continuous solution $\phi\leq\mu$ of \eqref{trav:eqn}.
\end{lem}

\begin{proof}
	Assume that $\phi\leq \mu$ is a nontrivial continuous solution of \eqref{trav:eqn}. The statement is a direct consequence of Lemma \ref{lem:distance}:  Since $\phi$ is continuous and has zero mean, we have that
	\[
	2(\mu-\|K_r\|_{L^1})\leq \phi(x_M)+\phi(x_m) <\phi(x_M) \leq \mu,
	\]
	where $\phi(x_M)= \max_{x\in\T} \phi(x)$ and $\phi(x_m)=\min_{x\in\T} \phi(x)<0$. Then
	\[
	\mu < 2 \|K_r\|_{L^1}.
	\]
\end{proof}


\section{Global bifurcation and conclusion of the main theorem}
\label{S:Global}

This section is devoted to the existence of nontrivial, even, periodic solutions of \eqref{trav:eqn}. After constructing small amplitude solutions via local bifurcation theory, we extend the local bifurcation branch globally and characterize the end of the global bifurcation curve. By excluding certain alternatives, based on a priori bounds on the wave speed (cf. Lemma \ref{lem:bound_mu} and Lemma \ref{lem:lowerbound} below), we prove that the global bifurcation curve reaches a limiting highest wave $\phi$, which is even, strictly monotone on its open half periods and with maximum at $\phi(0)=\mu$. By Theorem \ref{thm:reg} then, the highest wave is a peaked traveling-wave solution of
\[
u_t + L_ru_x + uu_x=0\qquad\mbox{for}\quad r>1.
\]

\medskip

We use the subscript $X_{\rm even}$ for the restriction of a Banach space $X$ to its subset of even functions.
 Let $\alpha \in (1,2)$ and set
\[
	F:C^{\alpha}_{0,\rm even}(\T)\times \mathbb{R}^{+}\rightarrow  C^{\alpha}_{0,\rm even}(\T),
\]
where
\begin{equation}\label{oper:F}
	F(\phi,\mu):= \mu\phi - L_r\phi - \phi^{2}/2+\widehat{\phi^2}(0)/2, \qquad (\phi, \mu) \in {{C}}^{\alpha}_{0,\rm even}(\T)\times \mathbb{R}_{+}.
\end{equation}
Then, $F(\phi,\mu)=0$ if and only if $\phi$ is an even $C^\alpha_0(\T)$-solution of \eqref{trav:eqn} corresponding to the wave speed $\mu\in \R_+$. Clearly. $F(0,\mu)=0$ for any $\mu\in \R_+$.
We are looking for $2\pi$-periodic, even, nontrivial solutions bifurcating from the line $\{(0,\mu)\mid \mu\in\R\}$ of trivial solutions. The wave speed $\mu>0$ shall be the bifurcation parameter.
The linearization of $F$ around the trivial solution $\phi=0$ is given by
\begin{equation*}\label{eq:Fderivative}
	D_\phi F(0,\mu): {{C}}^{\alpha}_{0,\rm even}(\T) \to {{C}}^{\alpha}_{0,\rm even}(\T), \qquad \phi\mapsto \left(\mu \,{\rm id}-L_r \right) \phi.
\end{equation*}
Recall that $L_r:{{C}}^{\alpha}_{0,\rm even}(\T) \to {\mathcal{C}}^{\alpha+r}_{0,\rm even}(\T)$ is parity preserving and a smoothing operator, which implies that it is compact on ${{C}}^{\alpha}_{0,\rm even}(\T)$. Hence, $D_\phi F(0,\mu)$ is a compact perturbation of an isomorphism, and therefore constitutes a Fredholm operator of index zero.

\medskip

The nontrivial kernel of $D_\phi F(0,\mu)$ is given by those functions $\psi\in  {{C}}^\alpha_{0,\rm even}(\T)$ satisfying
\begin{align*}
\widehat{\psi}(k)\left( \mu - |k|^{-r}\right) = 0,\ \ \ k\neq 0.
\end{align*}
For  $\mu\in (0,1]$, we see that $\supp\psi\subseteq \{\pm \mu^{-\frac{1}{r}}\}$.
Therefore, the kernel of $D_\phi F(0,\mu)$ is one-dimensional if and only if $\mu=|k|^{-r}$ for some $k\in \Z$, in which case it is given by
\begin{align*}\label{ker:form}
\ker D_\phi(0,\mu)= \mbox{span} \{\phi^*_k\} \qquad \mbox{with}\quad \phi^*_k(x):=\cos \left(xk \right).
\end{align*}
 The above discussion allows us to apply the Crandall--Rabionwitz theorem, where the transversality condition is trivially satisfied since we bifurcate from a simple eigenvalue (cf. \cite[Chapter 8.4]{buffoni2003}).

\begin{theo}[Local bifurcation]\label{cor:lcl:bfr}
For each integer $k\ge 1$, the point $(0,\mu_k^*)$, where $\mu_k^*=k^{-r}$ is a bifurcation point. More precisely, there exists $\e_0>0$ and an analytic curve through $(0,\mu^*_{k})$,
\begin{align*}
\{ (\phi_{k}(\e), \mu_{k}(\e))\mid |\e|<\e_0\} \subset {{C}}^{\alpha}_{0,\rm even}(\T)\times \R_+,
\end{align*}
of nontrivial, $\frac{2\pi}{k}$-periodic, even solutions of \eqref{oper:F} with $\mu_k(0)=\mu^*_k$ and 
\[
	D_{\e}\phi_{k}(0) =\phi^*_{k}(x)= \cos\left(xk\right).
\] In a neighborhood of the bifurcation point $(0,\mu^*_{k})$ these are all the nontrivial solutions of $F(\phi,\mu)=0$ in ${{C}}^{\alpha}_{0,\rm even}(\T)\times \R_+$.
\end{theo}

 We aim to extend the local bifurcation branch found in Theorem \ref{cor:lcl:bfr} to a global continuum of solutions of $F(\phi, \mu)=0$. Set
\begin{align*}
S:= \{(\phi,\mu)\in U: F(\phi,\mu)=0 \},
\end{align*}
where
\[
U:= \{(\phi,\mu)\in {C}^{\alpha}_{0,\rm even}(\T)\times \R_+\mid \ \phi<\mu\}.
\]


\begin{lem}\label{lem:glb:ind}
The Frech\'et derivative $D_{\phi}F(\phi,\mu)$ is a Fredholm operator of index $0$ for all $(\phi,\mu)\in U$.
\end{lem}
\begin{proof}
If $(\phi,\mu)\in U$, then $\phi<\mu$ and
\[
D_{\phi}F(\phi,\mu)=(\mu-\phi){\rm id}-L_r,
\]
constitutes a compact perturbation of an isomorphism. Thereby, it is a Fredholm operator of index zero.
\end{proof}
Let us recall that all bounded solutions $\phi$ of \eqref{trav:eqn}, that is all bounded solutions $\phi$ satisfying $F(\phi, \mu)=0$, are uniformly bounded by
\begin{equation}\label{eq:uniform_bound}
	\|\phi\|_\infty \leq 2(\mu + \|K_r\|_{L^1}+ 2\pi \|K_r^\prime\|_{L^1}),
\end{equation}
as shown in Lemma \ref{lem:uniform_bound}.
\begin{lem}\label{lem:cpt}
Any bounded and closed set of $S$ is compact in ${C}^{\alpha}_{0,\rm even}(\mathbb{T})\times \R_+$.
\end{lem}

\begin{proof}
 If $(\phi,\mu)\in S$, then in particular $\phi$ is smooth and
\begin{align*}\label{def:til:F}
\phi=\mu-\sqrt{\mu^{2} + \hat{\phi^2}(0)-2L_r\phi}=:\tilde{F}(\phi,\mu).
\end{align*}
Since the function $\tilde F$ maps $U$ into $\mathcal{C}^{\alpha +r}_{0, \rm even}(\T)$, the latter being compactly embedded into $C^\alpha_{0, \rm even}(\T)$, we obtain that $\tilde F$ maps bounded sets in $U$ into relatively compact sets in $C^{\alpha}_{0, \rm even}(\T)$. Let $A\subset S\subset U$ be a bounded and closed set. Then $\tilde F (A)=\{\phi \mid (\phi, \mu)\in A\}$ is relatively compact in $C^{\alpha}_{0, \rm even}(\T)$. In view of $A$ being closed, any sequence $\{(\phi_n,\mu_n)\}_{n\in \N}$ has a convergent subsequence in $A$. We conclude that $A$ is compact in $C^{\alpha}_{0, \rm even}(\T)\times \R_+$.
\end{proof}

Using Lemma~\ref{lem:glb:ind} and~\ref{lem:cpt} we can extend the local branches found in Theorem \ref{cor:lcl:bfr} to global curves. The result follows from \cite[Theorem 9.1.1]{buffoni2003} once we show that $\mu(\e)$ is not identically constant  for $0<\e\ll 1$.  The latter claim however is an immediate consequence of Theorem \ref{thm:Biformulas} below. The proof essentially follows  the lines in \cite[Section 4]{EK}. 
\begin{theo}[Global bifurcation]\label{thm:glb:bfr}
The local  bifurcation curve $s\mapsto (\phi_{k}(s),\mu_{k}(s))$  from Theorem~\ref{cor:lcl:bfr} of solutions of \eqref{oper:F} extends to a global continuous curve of solutions $\R_+\to S$ and one of the following alternatives holds:
\begin{itemize}
\item[(i)] $\|(\phi_{k}(s), \mu_{k}(s))\|_{C^{\alpha}(\mathbb{T})\times \R_+}$ is unbounded as $s \to \infty$.
\item[(ii)] The pair $(\phi_{k}(s),\mu_{k}(s))$ approaches the boundary of $S$ as $s\to \infty$.
\item[(iii)] The function $s\mapsto(\phi_{k}(s),\mu_{k}(s))$ is (finitely) periodic.
\end{itemize}
\end{theo}

\bigskip

\begin{center}
	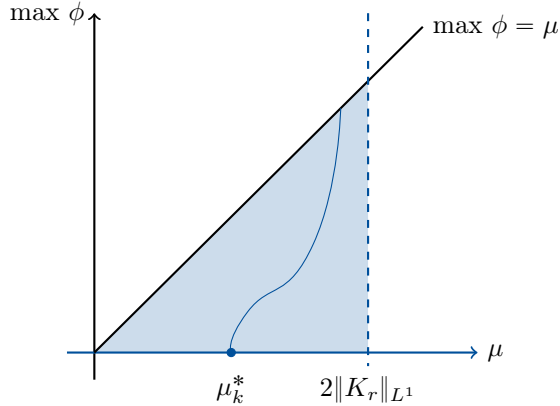
\begin{figure}[h]
	\begin{tikzpicture}[scale=1.8]
	\small
	\draw[->, thick] (0,-0.2) -- (0,2.5) node[left]{max $\phi$};
	\fill[luh-dark-blue!20] (0,0) -- (2,0) -- (2,2);
	\draw[luh-dark-blue,fill] (1,0) circle (0.03cm) node[below=5pt]{\textcolor{black}{$\mu_k^*$}};
	\draw[->,thick, luh-dark-blue] (-0.2,0) -- (2.8,0) node[right]{\textcolor{black}{$\mu$}};
	\draw[-, black, thick] (0,0) -- (2.4,2.4) node[right] {\textcolor{black}{max $\phi=\mu$}};
	\draw[dashed, luh-dark-blue, thick] (2,2.5) -- (2,-0.1) node[below]{\textcolor{black}{$2\|K_r\|_{L^1}$}} ;
	\draw [luh-dark-blue] plot [smooth,  tension=1] coordinates { (1,0) (1.12,0.3) (1.6,0.8) (1.8,1.8)};
	\end{tikzpicture}
	\caption{ Global bifurcation curve of $\frac{2\pi}{k}$-periodic, even solutions, reaching a limiting highest wave. Alternative (iii) in Theorem \ref{thm:glb:bfr} is excluded by Proposition \ref{prop:A3}. The dashed line indicates the natural upper bound on the wave speed found in Lemma \ref{lem:bound_mu}. Along the bifurcation branch the wave speed $\mu$ is bounded away from zero as shown in Lemma \ref{lem:lowerbound}. Eventually, alternative (i) and (ii)  in Theorem \ref{thm:glb:bfr} occur simultaneously, by Theorem \ref{thm:A12}.}
		\end{figure}
\end{center}

We apply the Lyapunov--Schmidt reduction, in order to establish the bifurcation formulas. Let $k\in \N$ be a fixed number and set
\[
	M:= \mbox{span}\left\{\cos\left(xl\right)\mid l\neq k\right\},\qquad N:=\ker D_\phi F(0,\mu^*_{k})= \mbox{span}\{\phi^*_{k}\}.
\]
Then, $C^\alpha_{0, \rm even}(\T)=M \oplus N$ and a continuous projection onto the one-dimensional space $N$ is given by
\[
	\Pi \phi = \left<\phi, \phi^*_{k}  \right>_{L_2}\phi^*_{k}
\]
where $\left< \cdot, \cdot \right>_{L_2}$ denotes the inner product in $L_2(\T)$. Let us recall the Lyapunov--Schmidt reduction theorem from \cite[Theorem I.2.3]{Kielhoefer}:

\begin{theo}[Lyapunov--Schmidt reduction] There exists a neighborhood $\mathcal{O}\times Y \subset U$ of $(0,\mu^*_{k})$ such that the problem
\begin{equation}
\label{eq:infinite}
	F(\phi, \mu)=0 \quad \mbox{for}\quad (\phi, \mu)\in \mathcal{O} \times Y
\end{equation}
is equivalent to the finite-dimensional problem
\begin{equation}
\label{eq:finite}
	\Phi(\e \phi^*_{k} , \mu):= \Pi F (\e \phi^*_{k}  + \psi(\e \phi^*_{k} , \mu), \mu)=0
\end{equation}
for functions $\psi \in C^\infty(\mathcal{O}_N \times Y, M)$ and $\mathcal{O}_N \subset N$ an open neighborhood of the zero function in $N$. One has that $\Phi(0, \mu^*_{k})=0$, $\psi(0,\mu^*_{k})=0$, $D_\phi \psi(0,\mu^*_{k})=0$, and solving problem \eqref{eq:finite} provides a solution
\[
	\phi= \e \phi^*_{k}+\psi (\e \phi^*_{k}, \mu)
\]
of the infinite-dimensional problem \eqref{eq:infinite}.
\end{theo}

\begin{theo}[Bifurcation formulas] 
	\label{thm:Biformulas}
	 The bifurcation curve found in Theorem \ref{thm:glb:bfr} satisfies
	\begin{equation}
	\label{eq:biformula1}
	\phi_k(\e)=\e \phi^*_k(x)- \frac{\e^2}{2}k^r\left( 1+\frac{1}{1-2^{-r}}\cos \left( 2kx\right)\right)+O(\e^3)
	\end{equation}
	and
	\begin{equation}
	\label{eq:biformula2}
	\mu_{k}(\e)=\mu^*_{k}+\e^2k^r\frac{3-2^{1-r}}{8(1-2^{-r})}+O(\e^3)
	\end{equation}
	in $C^\alpha_{0, \rm even}(\T)\times \R_+$ as $\e \to 0$. In particular, $\ddot{\mu}_{k}(0)>0$ for any $k\geq 1$, that is, Theorem \ref{cor:lcl:bfr} describes a supercritical pitchfork bifurcation.
\end{theo}

\begin{proof} Let us prove the bifurcation formula for $\mu_{k}$ first.
The value $\dot \mu_{k}(0)$ can be explicitly computed using the bifurcation formula
\[
	\dot \mu_{k} (0)=-\frac{1}{2}\frac{\left< D_{\phi \phi}^2 F(0,\mu^*_{k})[\phi^*_{k}, \phi^*_{k}], \phi^*_{k} \right>_{L^2}}{\left< D_{\phi \mu}^2 F(0,\mu^*_{k})\phi^*_{k},\phi^*_{k} \right>_{L^2}},
\]
cf. \cite[Section I.6]{Kielhoefer}. We have
\begin{align*}
	D_{\phi \phi}^2 F[0,\mu^*_{k}](\phi^*_{k},\phi^*_{k})&=(\phi^*_{k})^2,\\
	D_{\phi,\mu}^2 F[0,\mu^*_{k}]\phi^*_{k}&=-\phi^*_{k}.
\end{align*}
In view of $\int_{\T}(\phi^*_{k})^3(x)\,dx=0$, the first derivative of $\mu^*_{k}$ vanishes in zero. In this case the second derivative is given by
\begin{equation}
\label{eq:2derivative}
	\ddot \mu_{k}(0)=-\frac{1}{3}\frac{\left< D_{\phi\phi\phi}^3 \Phi(0,\mu^*_{k})[\phi^*_{k},\phi^*_{k},\phi^*_{k}],\phi^*_{k}\right>_{L_2}}{\left< D_{\phi \mu}^2 F(0,\mu^*_{k})\phi^*_{k},\phi^*_{k} \right>_{L^2}},
\end{equation}
where $\Phi \in C^\infty (\mathcal{O}_N \times Y, N)$ is the function defined in \eqref{eq:finite}. We have that
\begin{align*}
&D_\phi \Phi(\phi, \mu)\phi^*_{k}=\Pi D_\phi F(\phi+ \psi(\phi, \mu), \mu) \left[\phi^*_{k} + D_\phi \psi (\phi, \mu)\phi^*_{k} \right], \\
&D_{\phi\phi} \Phi (\phi, \mu)[\phi^*_{k},\phi^*_{k}] \\
&\quad=\Pi D_{\phi\phi}^2F(\phi + \psi(\phi, \mu), \mu)\left[\phi^*_{k} + D_\phi \psi(\phi,\mu)\phi^*_{k}, \phi^*_{k} + D_\phi \psi(\phi, \mu)\phi^*_{k} \right]\\
&\qquad + \Pi D_{\phi}F(\phi + \psi(\phi, \mu), \mu)D_{\phi \phi}^2\psi(\phi, \mu)[\phi^*_{k},\phi^*_{k}],\\
&D_{\phi\phi\phi}^3\Phi(\phi,\mu)[\phi^*_{k},\phi^*_{k},\phi^*_{k}]=  \Pi D_{\phi}F(\phi + \psi(\phi, \mu), \mu)D_{\phi\phi\phi}^3 \psi(\phi, \mu)[\phi^*_{k},\phi^*_{k},\phi^*_{k}]\\
&\qquad +3\Pi D_{\phi\phi}^2F(\phi+ \psi(\phi, \mu), \mu)[\phi^*_{k}+D_\phi\psi(\phi, \mu)\phi^*_{k},D^2_{\phi\phi}\psi(\phi,\mu)[\phi^*_{k},\phi^*_{k}]],
\end{align*}
in view of $F$ being quadratic in $\phi$ and therefore $D_{\phi\phi\phi}^3F(\phi,\mu)=0$. Using that $\psi(0,\mu^*_{k})=D_\phi \psi(0,\mu^*_{k})\phi^*_{k}=0$ we obtain that
\begin{align*}
D_{\phi\phi\phi}^3\Phi(0, \mu^*_{k})[\phi^*_{k},\phi^*_{k},\phi^*_{k}]&=\Pi D_\phi F(0,\mu^*_{k})D_{\phi\phi\phi}^3 \psi(0,\mu^*_{k})[\phi^*_{k},\phi^*_{k},\phi^*_{k}]\\
&\quad+3 \Pi D_{\phi\phi}^2 F(0,\mu^*_{k})[\phi^*_{k},D_{\phi\phi}^2 \psi(0,\mu^*_{k})[\phi^*_{k},\phi^*_{k}]].
\end{align*}
Since $N= \ker D_\phi F(0,\mu^*_{k})$ and $\Pi$ is the projection onto $N$, the above derivative reduces to 
\begin{align*}
D_{\phi\phi\phi}^3\Phi(0, \mu^*_{k})[\phi^*_{k},\phi^*_{k},\phi^*_{k}]=
3\Pi\phi^*_{k}D_{\phi\phi}^2 \psi(0,\mu^*_{k})[\phi^*_{k},\phi^*_{k}].
\end{align*}
As in \cite[Section 1.6]{Kielhoefer} we use that $D_\phi F(0,\mu^*_{k})$ is an isomorphism on $M$ to write
\begin{align}
\label{eq:Dphiphi}
\begin{split}
D_{\phi\phi}^2 \psi(0,\mu^*_{k})[\phi^*_{k},\phi^*_{k}]&=- (D_\phi F(0,\mu^*_{k}))^{-1}(1-\Pi)D_{\phi\phi}^2F(0,\mu^*_{k})[\phi^*_{k},\phi^*_{k}]\\
&=- (D_\phi F(0,\mu^*_{k}))^{-1}(1-\Pi)(\phi^*_{k})^2\\
&=-\frac{1}{2}(D_\phi F(0,\mu^*_{k}))^{-1}\left( 1+\cos\left(2xk\right)\right)\\
&=-\frac{1}{2}\left(\frac{1}{\mu^*_{k}}+\frac{\cos\left(2xk \right)}{\mu^*_{k}-(2k)^{-r}}\right).
\end{split}
\end{align}
We conclude that
\[
D_{\phi\phi\phi}^3\Phi(0, \mu^*_{k})[\phi^*_{k},\phi^*_{k},\phi^*_{k}]=-\frac{3}{2}\phi^*_{k} \left( \frac{1}{\mu^*_{k}}+\frac{1}{2(\mu^*_{k}-(2k)^{-r})} \right).
\] 
In view of the dominator in \eqref{eq:2derivative} being $-1$, the second derivative of $\mu_{k}$ at zero is given by
\begin{equation*}
\label{eq:2D}
	\ddot \mu_{k}(0)=\frac{1}{2}\left(\frac{1}{\mu^*_{k}}+\frac{1}{2(\mu^*_{k}-(2k)^{-r})} \right)= k^{r}\frac{3-2^{1-r}}{4(1-2^{-r})} >0,\ \ \text{for all}\ r>1.
\end{equation*}
The formula \eqref{eq:biformula2} is now a direct consequence of a Maclaurin series expansion and $\dot \mu_{k}(0)=0$.
Since $\ddot \mu_{k}(0)>0$, we conclude that the bifurcation curve describes a supercritical pitchfork bifurcation.

Keeping in mind that $\phi_k(0)=0$ and $\dot \phi_k(0)=\phi^*_k$, we are left to compute $\ddot \phi_k (0)$ in order to establish \eqref{eq:biformula1}. We use that
\[
\phi_k(\e)=\e\phi^*_k+\psi(\e\phi^*_k, \mu_k(\e)),
\]
cf. \cite[Chapter I.5]{Kielhoefer}. It follows that
\begin{align*}
	\ddot \phi_k(0)=&D^2_{\phi\phi}\psi(0,\mu^*_{k})[\phi^*_k,\phi^*_k]+2D^2_{\phi\mu}\psi(0,\mu^*_{k})[\phi^*_k,\dot \mu_{k}(0)] + D^2_{\mu\mu}\psi(0,\mu^*_{k})[\dot \mu_k(0),\dot \mu_{k}(0)]\\
	&+D_\mu\psi(0,\mu^*_{k})\dot \mu_{k}(0).
\end{align*}
Since $D_\mu \psi(0,\mu^*_{k})=0$ and $\dot \mu_{k}(0)=0$, we obtain that
\[
	\ddot \phi_k(0)=D^2_{\phi\phi}\psi(0,\mu^*_{k})[\phi^*_k,\phi^*_k].
\]
Thus, the claim follows from \eqref{eq:Dphiphi}. 
\end{proof}

\begin{lem}\label{lem:con}
	Any sequence of solutions $(\phi_{n},\mu_{n})_{n\in\N}\subset S$ to \eqref{trav:eqn} with $(\mu_n)_{n\in \N}$ bounded has a  subsequence which converges uniformly to a solution $\phi$.
\end{lem}
\begin{proof}
	In view of \eqref{eq:uniform_bound} the boundedness of  $(\mu_n)_{n\in \N}$ implies that also $(\phi_n)_{n\in \N}$ is uniformly bounded in $C(\T)$. In order to show that $(\phi_n)_{n\in \N}$ has a convergent subsequence, we prove that $(\phi_n)_{n\in \N}$ is actually uniformly H\"older continuous. By compactness, it then has a convergent subsequence in $C(\T)$. From Theorem \ref{thm:P} it is known  that $K_r$ is $\alpha$-H\"older continuous for some $\alpha\in (0,1]$. Since $(\phi_n)_{n\in \N}$ is uniformly bounded, we have that $(K_r*\phi_n)_{n\in\N}$ is uniformly $\alpha$-H\"older continuous. Recalling that
		\[
			\frac{1}{2}(\phi_n(x)-\phi_n(y))^2\leq |K_r*\phi_n(x)-K_r*\phi_n(y)|
		\]
		whenever $\phi_n\leq \mu$, we deduce that $(\phi_n)_{n\in\N}$ is uniformly $\frac{\alpha}{2}$-H\"older continuous. Thus, $(\phi_n,\mu_n)_{n\in \N}$ has a convergent subsequence which allows us to choose a uniformly convergent subsequences to a solution of~\eqref{trav:eqn}.
\end{proof}

The remainder of the section is devoted to exclude alternative (iii) in Theorem \ref{thm:glb:bfr} and to prove that alternative (i) and (ii) occur simultaneously, which in particular implies that the highest wave is reached as a limit of the global bifurcation curve.

\medskip

 Let
 \[
 \mathcal{K}_k:= \{ \phi\in {C}^{\alpha}_{0,\rm even}(\T):\ \phi\ \text{is $2\pi/k$-periodic and nondecreasing in}\ (-\pi/k,0)\},
 \]
 a closed cone in ${C}^{\alpha}_0(\mathbb{T})$.

\begin{prop}\label{prop:A3}
The solutions $\phi_{k}(s)$, $s>0$ on the global bifurcation curve belong to $\mathcal{K}_k\setminus \{0\}$ and alternative (iii) in Theorem \ref{thm:glb:bfr} does not occur. In particular, the bifurcation curve $(\phi_{k}(s),\mu_{k}(s))$ has no intersection with the trivial solution line for any $s>0$.
\end{prop}
\begin{proof}
Due to \cite[Theorem 9.2.2]{buffoni2003} the statement holds true if the following conditions are satisfied
\begin{itemize}
	\item [(a)] $\mathcal{K}_k$ is a cone in a real Banach space.
	\item[(b)] $(\phi_{k}(\e),\mu_{k}(\e))\subset \mathcal{K}_k\times \R$ provided $\e$ is small enough.
	\item[(c)] If $\mu \in \R$ and $\phi\in \ker D_\phi F(0,\mu)\cap \mathcal{K}_k$, then $\phi=\alpha \phi^*$ for $\alpha \geq 0$ and $\mu=\mu^*_{k}$.
	\item[(d)] Each nontrivial point on the bifurcation curve which also belongs to $\mathcal{K}_k\times \R$ is an interior point of $\mathcal{K}_k\times \R$ in $S$.
\end{itemize} 
In view of the local bifurcation result in Theorem \ref{cor:lcl:bfr}, we are left to verify condition (d). Let $(\phi,\mu)\in \mathcal{K}_k\times \R$ be a nontrivial solution on the bifurcation curve found in Theorem \ref{thm:glb:bfr}.
 By Theorem~\ref{th:phi:prop}, $\phi$ is smooth and together with Lemma~\ref{lem:nod}, we have that $\phi'>0$ on $(-\pi,0)$ and $\phi''(0)<0$. Choose a solution $\varphi$ lying within a $\delta \ll1$ small enough neighborhood in $C^\alpha_0(\T)$ such that $\varphi < \mu$ and $\|\phi-\varphi\|_{C^{\alpha}}<\delta$. In view of \eqref{maps:reg} an iteration process on the regularity index yields that $\|\phi-\varphi\|_{C^{2}}<\tilde{\delta}$, where $\tilde{\delta}>0$ depends on $\delta$ and can be made arbitrarily small by choosing $\delta$ small enough. It follows that for $\delta$ small enough $\varphi<\mu$ is a smooth, even solution, nondecreasing on $(-\frac{\pi}{k},0)$ and hence $(\phi,\mu)$ belongs to the interior of $\mathcal{K}_k\times \R$ in $S$, which concludes the proof.
\end{proof}

\begin{lem}\label{lem:lowerbound}
Along the bifurcation curve in Theorem \ref{thm:glb:bfr} we have that
	\[
		\mu(s)\gtrsim 1
	\]
	uniformly for all $s\geq0$.
	\end{lem}
	\begin{proof}
		Let us assume for a contradiction that there exists a sequence $(s_n)_{n\in \N}\in \R_+$ with $\lim_{n\to \infty}s_n=\infty$ such that $\mu(s_n)\to 0$ as $n\to \infty$, while $\phi(s_n)\to \phi_0$ as $n\to \infty$ along the bifurcation curve found in Theorem \ref{thm:glb:bfr}. In view of Lemma \ref{lem:con}, there exists a subsequence of $(s_n)_{n\in \N}$ (not relabeled) such that $\phi(s_n)$ converges to a solution $\phi_0$ of \eqref{oper:F}. Along the bifurcation curve we have that $\phi(s_n)<\mu(s_n)$. Taking into account the zero mean property of solutions of \eqref{oper:F}, it follows that $\phi_0=0$ is the trivial solution. But then Lemma \ref{lem:lowerbound:aux} yields the contradiction
		\[
			0=\lim_{n\to \infty}\left(\mu(s_n)-\phi(s_n)(\pi)\right)\geq \lambda \pi>0.
		\]
	\end{proof}

\begin{theo} \label{thm:A12}
	In Theorem \ref{thm:glb:bfr},  alternative (i) and  (ii) both occur. 
\end{theo}

\begin{proof}
Let $(\phi_{k}(s),\mu_{k}(s)), s\in\R$, the bifurcation curve found in Theorem \ref{thm:glb:bfr}. In view of Proposition  \ref{prop:A3} we know that any solution along the bifurcation curve is even and nondecreasing on $(-\frac{\pi}{k},0)$. Moreover,  alternative (iii) in Theorem \ref{thm:glb:bfr} is excluded. That is either alternative (i) or alternative (ii) in Theorem \ref{thm:glb:bfr} occur. Let us assume first that alternative (i) occurs, that is either $\|\phi_{k}(s)\|_{C^\alpha}\to \infty$  for some $\alpha\in (1,2)$ or $|\mu_{k}(s)|\to \infty$ as $s\to \infty$. The former case implies alternative (ii) in view of Theorem \ref{th:phi:prop}.
Since $\phi_{k}(s)$ has zero mean and keeping in mind Lemma \ref{lem:bound_mu}, it is clear that the second option $\lim_{s\to \infty}|\mu_{k}(s)|=\infty$ can not happen unless we reach the trivial solution line, which is excluded by Proposition \ref{prop:A3}. Suppose now that alternative (ii) occurs, but not alternative (i). Then there exists a sequence $(\phi_{k}(s_n),\mu_{k}(s_n))_{n\in \N}$ in $S$ satisfying $\phi_{k}(s_n)<\mu$ and $\lim_{n\to \infty}\max \phi_{k}(s_n)=\mu$, while $\phi_{k}(s_n)$ remains uniformly  bounded in $C^\alpha(\T)$ for $\alpha\in (1,2)$ and $\mu\gtrsim 1$ by Lemma \ref{lem:lowerbound}. But this is clearly a contradiction to Theorem \ref{thm:reg}. We deduce that both, alternative (i) and alternative (ii) occur simultaneously.
\end{proof}

Now, we are at the end of our analysis and conclude our main result: Let $(\phi_{k}(s),\mu_{k}(s))$ be the global bifurcation curve found in Theorem \ref{thm:glb:bfr} and let $(s_n)_{n\in \N}$ be a sequence in $\R_+$ tending to infinity. Due to our previous analysis (Lemma \ref{lem:bound_mu} and Proposition \ref{prop:A3}), we know that $(\mu_{k}(s_n))_{n\in \N}$ is bounded and bounded away from zero. In view of the $\mu_{k}$-dependent bound of $\phi_{k}$ we obtain that also $(\phi_{k}(s_n))_{n\in \N}$ is bounded, whence Lemma \ref{lem:con} implies the existence of a converging subsequence (not relabeled) of $(\phi_{k}(s_n),\mu_{k}(s_n))_{n\in\N}$. Let us denote the limit by $(\bar \phi, \bar \mu)$. By Theorem \ref{thm:A12} and Theorem \ref{thm:reg} we conclude that $\bar \phi (0)=\bar \mu$ with $\bar \phi$ admitting precisely Lipschitz regularity at each crest, which proves the main assertion Theorem \ref{thm:main}.

\bigskip

\section{Application to the reduced Ostrovsky equation}
\label{S:RO}

In this section we show that our approach can be applied to traveling-waves of the reduced Ostrovsky equation, which is given by
\begin{equation}\label{eq:GRO}
	\left[u_t +uu_x \right]_x-u=0
\end{equation}
and arises in the context of long surface and internal gravity waves in a rotating fluid~\cite{Ostrovsky1978}.

We are looking for $2\pi$-periodic traveling-wave solutions $u(t,x)=\phi(x-\mu t)$, where $\mu>0$ denotes the speed of the right-propagating wave. In this context equation \eqref{eq:GRO} reduces to
\begin{equation}\label{eq:T}
	\left[ \frac{1}{2}\phi^{2}-\mu\phi\right]_{xx}-\phi=0.
\end{equation}

Let us emphasize that the existence of periodic traveling wave solutions of \eqref{eq:GRO} is well-known. Furthermore, there exists an explicit example of a $2\pi$-periodic traveling-wave with wave speed $\mu=\frac{\pi^2}{9}$ of the form
\begin{equation}\label{eq:formula}
	\phi_p(x)=\frac{3x^2-\pi^2}{18},
\end{equation}
which satisfies \eqref{eq:T} point-wise on $(-\pi,\pi)$. It is easy to check that $\phi_p$ is precisely Lipschitz continuous at its crest points located at $\pi(2\Z+1)$ and smooth elsewhere.

	\begin{center}
		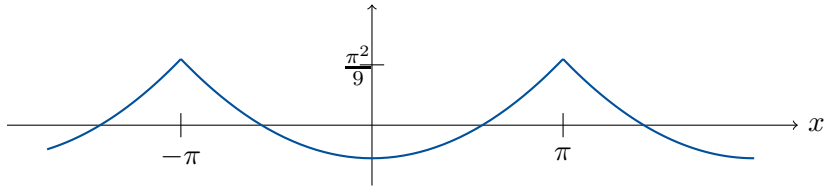
\begin{figure}[h]
		\begin{tikzpicture}[scale=0.8]
		\draw[->] (-6,0) -- (7,0) node[right] {$x$};
		\draw[->] (0,-1) -- (0,2);
		\draw[-] (-0.2, 1)--(0.2,1) node[left]{$\frac{\pi^2}{9}$};
		\draw[-] (-pi, 0.2)--(-pi,-0.2) node[below]{$-\pi$};
		\draw[-] (pi, 0.2)--(pi,-0.2) node[below]{$\pi$};
		\draw[domain=-pi:pi,smooth,variable=\x,luh-dark-blue, thick] plot ({\x},{(3*\x*\x-pi^2)/18});
		\draw[domain=-1.7*pi:-pi,smooth,variable=\x,luh-dark-blue, thick] plot ({\x},{((3*(\x+2*pi)*(\x+2*pi)-pi^2)/18});
		\draw[domain=pi:2*pi,smooth,variable=\x,luh-dark-blue, thick] plot ({\x},{((3*(\x-2*pi)*(\x-2*pi)-pi^2)/18});
		\end{tikzpicture}
		\caption{Explicit $2\pi$-periodic, peaked traveling-wave solution for the reduced Ostrovsky equation \eqref{eq:GRO} obtained via formula \eqref{eq:formula} by periodic extension.}
		\end{figure}
	\end{center}


Recall that any periodic solution of \eqref{eq:T} has necessarily zero mean. Therefore, working in suitable spaces restricted to their zero mean functions, the  pseudo differential operator $\partial_x^{-2}$ can be defined uniquely in terms of a Fourier multiplier.
We show in Lemma \ref{lem:relation} that the steady reduced Ostrovsky equation \eqref{eq:T} can be reformulated in nonlocal form as
	\begin{equation}
	\label{eq:Reduced_Ostrovsky}
	-\mu \phi + L\phi + \frac{1}{2}\left( \phi^{2}-\widehat{\phi^2}(0)\right)=0.
	\end{equation}
Here $L$ denotes the Fourier multiplier with symbol $m(k)=k^{-2}$ for $k\neq 0$ and $m(0)=0$.

\medskip

Recall that any function $f\in {C}^\alpha(\T)$ for $\alpha>\frac{1}{2}$ has an absolutely convergent Fourier series, that is
\[
	\sum_{k\in \Z}|\hat f(k)|<\infty,
\]
 and the Fourier representation of $f$ is given by
\[
		f(x)=\sum_{k\in\Z}\hat f\left(k\right)e^{ixk}.
\]
%
\begin{lem}\label{lem:relation} Let $\alpha>\frac{1}{2}$. A function $\phi\in {C}_0^\alpha (\T)$ is a solution of \eqref{eq:T} if and only if $\phi$ solves
\[
	-\mu \phi + L_2\phi + \frac{1}{2}\left( \phi^{2}-\widehat{\phi^2}(0)\right)=0,
\]
where 
\[
 L\phi(x) :=\sum_{k\neq 0}k^{-2}\hat \phi(k)e^{ixk}.
\]
\end{lem} 

\begin{proof} Notice that $\phi\in \mathcal{C}_0^\alpha(\T)$ is a solution of \eqref{eq:T} if and only if 
\[
	\int_{-\pi}^\pi  \left[ \frac{1}{2}\phi^{2}(x)-\mu\phi(x)\right]\psi_{xx}(x)\,dx = \int_{-\pi}^\pi\phi(x) \psi(x)\,dx
\]
for all $\psi \in C_c^\infty(-\pi,\pi)$, which is equivalent to
\[
	\mathcal{F}\left( \left[ \frac{1}{2}\phi^{2}-\mu\phi\right]\psi_{xx}\right)(0)= \mathcal{F}\left(\phi \psi\right)(0).
\]
 Using the property that the Fourier transformation translates products into convolution, we can write
\[
\mathcal{F}\left( \frac{1}{2}\phi^{2}-\mu\phi\right) * \mathcal{F}\left(\psi_{xx}\right)(0)=\hat \phi * \hat \psi(0).
\]
In view of $\phi$ having zero mean and therefore $\hat \phi (0)=0$, we deduce that $\phi\in {C}_0^\alpha (\T)$ is a solution to \eqref{eq:T} if and only if
\[
	-\sum_{k\neq 0}\mathcal{F}\left( \frac{1}{2}\phi^{2}-\mu\phi\right)(-k)k^2\hat \psi (k)=\sum_{k\neq 0 }\hat \phi(-k)\hat \psi (k)
\]
for all $\psi \in C_c^\infty(-\pi,\pi)$. In particular,
\[
	 \frac{1}{2}\widehat{\phi^{2}}(k)-\mu\hat \phi(k)+k^{-2}\hat\phi (k)=0 \qquad \mbox{for all}\quad k \neq 0,
\]
which is equivalent to
\[
\sum_{k\neq 0} \left( \frac{1}{2}\widehat{\phi^{2}}(k)-\mu\hat \phi(k)+k^{-2}\hat\phi (k)\right)e^{ixk}=0.
\]
Due to the fact that $\phi$ has zero mean, the above equation can be rewritten as
\[
	-\mu \phi + L\phi + \frac{1}{2}\left( \phi^{2}-\widehat{\phi^2}(0)\right)=0,
\]
which proves the statement.

\end{proof}

We proved in Theorem \ref{thm:reg} that \emph{any} even, periodic, bounded solution $\phi\leq \mu$, which is monotone on a half period, is Lipschitz continuous on $\R$, which guarantees by Lemma~\ref{lem:relation} that all solutions of \eqref{eq:Reduced_Ostrovsky} we consider here are indeed solutions of the reduced Ostrovsky equation.

\medskip

As a consequence of our main result Theorem~\ref{thm:main}, we obtain the following corollary:
\begin{cor}[Highest wave for the reduced Ostrovsky equation]\label{cor:RO}
For each integer $k\geq 1$ there exists  a global bifurcation branch
	\[
	s\mapsto (\phi_{k}(s),\mu_{k}(s)),\qquad s>0,
	\]
	of nontrivial, $\frac{2\pi}{k}$-periodic, smooth, even solutions to the steady reduced Ostrovsky equation \eqref{eq:T} emerging from the bifurcation point $(0,k^{-2})$. Moreover, given any unbounded sequence $(s_n)_{n\in\N}$ of positive numbers $s_n$, there exists a subsequence of $(\phi_{k}(s_n))_{n\in \N}$, which converges uniformly to a limiting traveling-wave solution $(\bar \phi_{k},\bar\mu_{k})$ that solves \eqref{eq:T} and satisfies
	\[
	\bar \phi_{k}(0)=\bar \mu_{k}.
	\]
	The limiting wave is strictly increasing on $(-\frac{\pi}{k},0)$ and is exactly Lipschitz at $x\in \frac{2\pi}{k}\Z$.
\end{cor}

In case of the reduced Ostrovsky equation, we know even more about the bifurcation diagram.
  Using methods from dynamical systems, the authors of \cite{GP, GP2} are able to prove that the peaked, periodic traveling-wave \eqref{eq:formula} for the reduced Ostrovsky equation is the \emph{unique} nonsmooth $2\pi$-periodic traveling-wave solution (\cite[Lemma 2]{GP2}). Moreover, from \cite[Lemma 3]{GP} we obtain the following a priori bound on the wave speed for nontrivial, $2\pi$-periodic traveling-wave solutions  of \eqref{eq:GRO}:
  
  \begin{lem}[\cite{GP}, Lemma 3]\label{lem:optimal}
  	If $\phi$ is a nontrivial, smooth, $\frac{2\pi}{k}$-periodic, traveling-wave solution of the reduced Ostrovsky equation, then the wave speed $\mu$ satisfies the bound
  	\[
  	\mu\in k^{-2}\left(1,\frac{\pi^2}{9}\right).
  	\]
  \end{lem}

\begin{center}
	\begin{figure}[h]
		\begin{tikzpicture}[scale=1.8]
		\small
		\draw[->, thick] (0,-0.2) -- (0,2.5) node[left]{max $\phi$};
		\fill[luh-dark-blue!20] (1,0) -- (1,1) -- (1.8,1.8) -- (1.8,0);
		\draw[luh-dark-blue,fill] (1,0) circle (0.03cm) node[below=5pt]{\textcolor{black}{$k^{-2}$}};
		\draw[->,thick, luh-dark-blue] (-0.2,0) -- (2.8,0) node[right]{\textcolor{black}{$\mu$}};
		\draw[-, black, thick] (0,0) -- (2.4,2.4) node[right] {\textcolor{black}{max $\phi=\mu$}};
		\draw[dashed, luh-dark-blue] (1.8,1.8) -- (1.8,-0.1) node[below]{\textcolor{black}{$k^{-2}\frac{\pi^2}{9}$}} ;
		\draw[dashed, luh-dark-blue] (1,1) -- (1,0) ;
		\draw [luh-dark-blue] plot [smooth,  tension=1] coordinates { (1,0) (1.12,0.3) (1.6,0.8) (1.8,1.8)};
		\end{tikzpicture}
		\caption{ Sketch for the global bifurcation diagram for $\frac{2\pi}{k}$-periodic, even solutions of the reduced Ostrovsky equation reaching the unique limiting highest wave. Outside of the shaded region, there exist no nontrivial $\frac{2\pi}{k}$-periodic, smooth traveling-wave solutions of \eqref{eq:GRO}.}
	\end{figure}
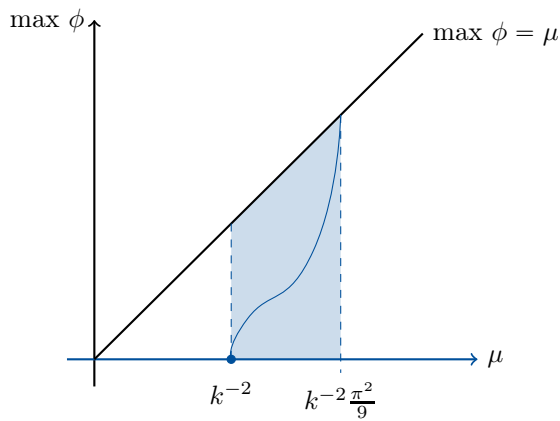
\end{center}

\begin{rem}
\emph{Notice, that in the class of $2\pi$-periodic solutions, the range for the wave speed $\mu$ supporting nontrivial traveling-wave solutions of the reduced Ostrovsky equation is given by $(1,\frac{\pi^2}{9})$, where $\mu=1$ is the wave speed from which nontrivial, $2\pi$-periodic solutions bifurcate and $\mu=\frac{\pi^2}{9}$ is exactly the wave speed corresponding to the highest peaked wave in \eqref{eq:formula}.}
\end{rem}

\begin{rem}
\emph{
	Regarding the $2\pi$-periodic, nontrivial traveling-wave solutions of \eqref{eq:GRO} on the global bifurcation branch from Corollary \ref{cor:RO}, we have that Lemma~\ref{lem:bound_mu} and Lemma \ref{lem:lowerbound}, proved in the previous sections, guarantee that the wave speed is a priori bounded by
	\[
		\mu \in \left(M, \frac{4\pi^3}{9\sqrt{3}}\right)\qquad\mbox{for some}\quad M\in (0,1].
	\]
	Certainly this bound is if far from the optimal bound provided by \cite{GP} in Lemma~\ref{lem:optimal}. Thus, there is still room for improvement in our estimates.
	}
\end{rem}

\bigskip

\subsection*{Acknowlegments}
The author G.B. would like express her gratitude to Mats Ehrnstr\"om for many valuable discussions. Moreover, G.B. gratefully acknowledges financial support by the Deutsche Forschungsgemeinschaft (DFG) through CRC 1173.
Part of this research was carried out while G.B. was supported by grant no. 250070 from the Research Council of Norway.\\ 
The author R.N.D. acknowledges the support during the tenure of an ERCIM `Alain Bensoussan' Fellowship Program and was supported by grant nos. 250070 \& 231668 from the Research Council of Norway. Moreover, R.N.D. would also like to thank the Fields Institute for
Research in Mathematical Sciences for its support to attain the Focus Program on \emph{Nonlinear Dispersive
Partial Differential Equations and Inverse Scattering} (July 31 to August 23, 2017) in the related field. Its contents are solely the responsibility of the authors and do not necessarily represent the official views of the  \href{www.fields.utoronto.ca}{Fields Institute}.

\bigskip

%
\bibliographystyle{plain}

\bibliography{BD_Reduced_Ostrovsky}
\end{document}